\documentclass[11pt]{amsart}
\usepackage[parfill]{parskip}    
\usepackage{graphicx, txfonts}
\usepackage{epstopdf}
\usepackage{hyperref}
\usepackage{amsthm, amsmath, amssymb, amsfonts, amsmath, enumerate, xy, mathrsfs}
\input xy
\xyoption{all}
\newtheorem{theorem}{Theorem}[section]
\newtheorem{proposition}[theorem]{Proposition}
\newtheorem{prop}[theorem]{Proposition}
\newtheorem{lemma}[theorem]{Lemma}
\newtheorem{corollary}[theorem]{Corollary}

\theoremstyle{definition}

\newtheorem{defn}[theorem]{Definition}
\newtheorem{conjecture}[theorem]{Conjecture}

\newtheorem{example}[theorem]{Example}

\theoremstyle{remark}
\newtheorem{remark}[theorem]{Remark}

\renewcommand{\P}{\mathcal{P}}

\newcommand{\N}{\mathbb{N}}
\newcommand{\Q}{\mathbb{Q}}

\newcommand{\F}{\mathbb{F}}

\newcommand{\M}{\mathcal{M}}
\newcommand{\C}{\mathcal{C}}

\newcommand{\D}{\mathcal{D}}

\newcommand{\sP}{\mathscr{P}}

\newcommand{\sW}{\mathscr{W}}
\newcommand{\sQ}{\mathscr{Q}}

\renewcommand{\emptyset}{\varnothing}

\renewcommand{\tilde}[1]{\widetilde{#1}}

\DeclareMathOperator{\colim}{colim}

\DeclareMathOperator{\Arr}{Arr}
\DeclareMathOperator{\Sym}{Sym}
\DeclareMathOperator{\CMon}{CMon}
\DeclareMathOperator{\skel}{skel}

\DeclareMathOperator{\Mon}{Mon}

\DeclareMathOperator{\ho}{Ho}

\newcommand{\po}{\ar@{}[dr]|(.7){\Searrow}}
\newcommand{\pb}{\ar@{}[dr]|(.3){\Nwarrow}}

\newcommand{\cat}[1]{\mathcal{#1}}

\newcommand{\boxprod}{\mathbin\square}

\begin{document}
\title{Model Structures on Commutative Monoids in General Model Categories}

\author{David White}
\address{Denison University
\\ Granville, OH 43023}
\email{david.white@denison.edu}

\begin{abstract}
We provide conditions on a monoidal model category $\mathcal{M}$ so that the category of commutative monoids in $\mathcal{M}$ inherits a model structure from $\mathcal{M}$ in which a map is a weak equivalence or fibration if and only if it is so in $\mathcal{M}$. We then investigate properties of cofibrations of commutative monoids, rectification between $E_\infty$-algebras and commutative monoids, the relationship between commutative monoids and monoidal Bousfield localization functors, when the category of commutative monoids can be made left proper, and functoriality of the passage from a commutative monoid $R$ to the category of commutative $R$-algebras. In the final section we provide numerous examples of model categories satisfying our hypotheses.

\end{abstract}

\maketitle

\section{Introduction}
In recent years, the importance of monoidal model categories has been demonstrated by a number of striking results related to structured (equivariant) ring spectra, c.f. \cite{EKMM}, \cite{hovey-shipley-smith}, \cite{kervaire-arxiv}, \cite{MMSS}, \cite{toen-vezzosi}. Commutative monoids played a key role in many of these applications, and it became important to have a model structure on objects with commutative structure, compatibly with the monoidal model structure on the underlying category $\M$.

The non-commutative case was treated in \cite{SS00}, where the authors introduced the monoid axiom. They prove that if $\M$ satisfies the monoid axiom then the category of monoids in $\M$ inherits a model structure from $\M$ with weak equivalences (resp. fibrations) maps that are weak equivalences (resp. fibrations) in $\M$. They then verify that the monoid axiom holds for all examples of interest. 

In this paper we will take a similar approach and introduce the commutative monoid axiom, which guarantees us that commutative monoids in $\M$ inherit a model structure. In \cite{SS00}, the authors refer to the commutative situation as ``intrinsically more complicated'' and indeed there are several known cases where commutative monoids cannot inherit a model structure in the way above, e.g. commutative differential graded algebras over a field of nonzero characteristic, $\Gamma$-spaces, and non-positive model structures on symmetric or orthogonal spectra (due to an example of Gaunce Lewis in \cite{lewis}). Side-stepping Lewis's example required the introduction of positive variants on diagram spectra in \cite{MMSS}, and the convenient model structure on symmetric spectra introduced in \cite{shipley-positive} (nowadays referred to as the positive flat model structure). We discuss these examples in Section \ref{sec:examples}.

One way to get around these obstacles is to work with $E_\infty$-algebras everywhere and never ask for strict commutativity. It is much easier to place a model structure on $E_\infty$-algebras because $E_\infty$ operads are $\Sigma$-cofibrant, while $Com$ is not. We feel it is important to also be able to treat the strict commutative case, because outside of categories of structured ring spectra one does not know that there is a Quillen equivalence between $E_\infty$-algebras and strictly commutative monoids (because $Com$ is not $\Sigma$-cofibrant, one cannot use the general rectification results in \cite{BM03}). The crucial hypothesis which allows such a Quillen equivalence in the case of structured ring spectra is that for all cofibrant $X$, the map $(E\Sigma_n)_+ \wedge_{\Sigma_n} X^{\wedge n} \to X^{\wedge n}/\Sigma_n$ is a weak equivalence. It is important to note that this hypothesis is not necessary for strictly commutative monoids to inherit a model structure (in particular, it fails for simplicial sets). This hypothesis appears to be more related to the rectification question than to the question of existence of model structures. We address the point further in Section \ref{subsec:rectification}.

Due to the difficulties associated with passing model structures to categories of commutative monoids, several important papers have folded the existence of a model structure on commutative monoids into their hypotheses. This is done in Assumption 1.1.0.4 in \cite{toen-vezzosi} and in Hypothesis 5.5 in \cite{shipley-uniqueness}, among other places. The results in Section \ref{sec:main} provide checkable conditions on $\M$ so that those hypotheses are satisfied.

We remark that a different axiom on $\M$ which guarantees commutative monoids inherit a model structure has appeared as Proposition 4.3.21 in \cite{DAG3}. However, it is pointed out in \cite{MO-error-Lurie} that this work contains some errors and as written does not apply to the positive model structure on symmetric spectra. Furthermore, we will demonstrate that it does not apply to topological spaces, though it does apply to chain complexes over a field of characteristic zero. Our commutative monoid axiom is more general, and does apply in these situations.

After a review of model category preliminaries in Section \ref{sec:background}, we will proceed to state the commutative monoid axiom and prove our main result in Section \ref{sec:main}, highlighting differences from the situation of \cite{DAG3} as we go. We additionally discuss when a cofibration of commutative monoids forgets to a cofibration in $\M$, and we introduce the strong commutative monoid axiom to guarantee this occurs. Following \cite{SS00}, we place the details of the proofs of these main results in Appendix \ref{appendix-proof} and we also prove in Appendix \ref{appendix-generators} that it is sufficient to check the strong commutative monoid axiom on the generating (trivial) cofibrations. Using this, we collect examples in Section \ref{sec:examples}. We include additional results regarding functoriality of the passage from $R$ to commutative $R$-algebras, regarding rectification between $Com$ and $E_\infty$, regarding the interplay between the strong commutative monoid axiom and Bousfield localization, and regarding left properness for the category of commutative monoids in Section \ref{sec:additional-results}. Finally, we conclude with a discussion in Appendix \ref{sec:operads} of what can be said for operads other than $Com$.

\section{Preliminaries}
\label{sec:background}

We assume the reader is familiar with basic facts about model categories. Excellent introductions to the subject can be found in \cite{dwyer-spalinski}, \cite{hovey-book}, or \cite{hirschhorn}. Throughout the paper we will assume $\M$ is a cofibrantly generated model category, i.e. there is a set $I$ of cofibrations and a set $J$ of trivial cofibrations which permit the small object argument (with respect to some cardinal $\kappa$), and a map is a (trivial) fibration if and only if it satisfies the right lifting property with respect to all maps in $J$ (resp. $I$). 

A morphism $f$ is a relative $I$-cell complex if $f$ is a transfinite composition of pushouts of elements of $I$, i.e. a transfinite composition of morphisms $f_\alpha:X_\alpha \to X_{\alpha+1}$ where each $f_\alpha$ is obtained as the pushout of a map $g_\alpha:C_\alpha \to D_\alpha$ in $I$ along a map $C_\alpha \to X_\alpha$. Let $I$-cell denote the class of relative $I$-cell complexes, let $I$-cof denote the class of morphisms that are retracts of relative $I$-cell complexes, and let $I$-inj denote the class of morphisms that have the right lifting property with respect to $I$. 

In order to run the small object argument, we will assume the domains $K$ of the maps in $I$ (and $J$) are $\kappa$-small relative to $I$-cell (resp. $J$-cell), i.e. given a regular cardinal $\lambda \geq \kappa$, any $\lambda$-sequence $X_0\to X_1\to \dots$ formed of maps $X_\beta \to X_{\beta+1}$ in $I$-cell, then the map of sets $\colim_{\beta < \lambda} \M(K,X_\beta) \to \M(K,\colim_{\beta < \lambda} X_\beta)$ is a bijection. An object is small if there is some $\kappa$ for which it is $\kappa$-small. See Chapter 10 of \cite{hirschhorn} for a more thorough treatment of this material. For any object $X$ we have a cofibrant replacement $QX \to X$ and a fibrant replacement $X\to RX$.

A monoidal model category is a model category $\M$ that is also a closed symmetric monoidal category with product $\otimes$ and unit $S \in \M$. The closed assumption guarantees that $X\otimes -$ is a left adjoint (hence preserves colimits), and this is needed so that $X\otimes -$ can be a left Quillen functor when $X$ is cofibrant. In order to ensure that the monoidal structure interacts nicely with the model structure (e.g. to guarantee it passes to a monoidal structure on the homotopy category $\ho(\M)$ whose unit is given by $S$) we must assume two conditions:

\begin{enumerate}
  \item Unit Axiom: For any cofibrant $X$, the map $QS\otimes X \to S\otimes X\cong X$ is a weak equivalence.
  \item Pushout Product Axiom: Given any cofibrations $f:X_0\to X_1$ and $g:Y_0\to Y_1$, the map $f\boxprod g: X_0\otimes Y_1 \coprod_{X_0\otimes Y_0}X_1\otimes Y_0 \to X_0\otimes Y_0$ is a cofibration. Furthermore, if either $f$ or $g$ is trivial then $f\boxprod g$ is trivial.
\end{enumerate}

If these hypotheses are satisfied then $\M$ is called a \textit{monoidal model category}. Note that the pushout product axiom forces $\otimes$ to be a Quillen bifunctor. Furthermore, it is sufficient to check the pushout product axiom on the generating maps $I$ and $J$, by Lemma 3.5 in \cite{SS00}. Lastly, it is worth remarking that none of our proofs require the unit axiom, so our results remain valid if one uses the definition of monoidal model category from \cite{SS00}, which does not require the unit axiom.

We turn now to the problem of placing model structures on categories of algebras. Let $\P$ be an operad valued in $\M$. For this discussion it will be fine to think of $\P$ as $Ass$ or $Com$. For a general discussion of operads see \cite{BM03}. A $\P$-algebra is an object $X$ of $\M$ with an action of $\P$ encoded by maps $\P(n)\otimes A^{\otimes n} \to A$ for all $n\geq 0$ satisfying $\Sigma_n$-equivariance, associativity, and unit conditions. Let $\P$-alg denote the category whose objects are $\P$-algebras and whose morphisms are $\P$-algebra homomorphisms (i.e. respect the $\P$-action). 

Let $P: \M \to \P$-alg be the free $\P$-algebra functor and let $U:\P$-alg$\to \M$ be the forgetful functor. Then $(P,U)$ is an adjoint pair. When $\P$ is $Ass$, the free monoid functor $X \mapsto S \coprod X \coprod X^{\otimes 2} \coprod \dots$ has been known to topologists for years as the James construction. When $\P$ is $Com$, the free commutative monoid functor $X\mapsto S \coprod X \coprod X^{\otimes 2}/\Sigma_2 \coprod \dots$ is sometimes called the $SP^\infty$ functor, or the Dold-Thom functor.

In order for there to be a model structure on $\P$-alg which is compatible with the model structure on $\M$, it must be the model structure which is transferred across the pair $(P,U)$ so that $(P,U$) forms a Quillen pair. In particular, a weak equivalence or fibration of $\P$-algebras will be a map which is a weak equivalence or fibration in $\M$. If such a model structure on $\P$-alg exists, we say it is \textit{inherited} from $\M$. Proving the existence of this model structure comes down to Lemma 2.3 in \cite{SS00}:

\begin{lemma} \label{ss00-lemma2.3}
Suppose $\M$ is cofibrantly generated and $T$ is a monad which commutes with filtered direct limits. If the domains of $T(I)$ and $T(J)$ are small relative to $T(I)$-cell and $T(J)$-cell respectively and 

\begin{enumerate}
\item $T(J)-$cell$\subset \sW$, or
\item All objects are fibrant and every $T$-algebra has a path object (factoring $\delta:X \to X\otimes X$ into $\stackrel{\simeq}{\hookrightarrow} \twoheadrightarrow$)
\end{enumerate}

then $T$-alg inherits a cofibrantly generated model structure with fibrations and weak equivalences created by the forgetful functor to $\M$.
\end{lemma}

When the conditions of this lemma are satisfied, $\P$-alg inherits a cofibantly generated model structure in which $P(I)$ and $P(J)$ are the generating (trivial) cofibrations. The case $\P = Ass$ was treated in \cite{SS00} and checking the first condition of the lemma led to the introduction of the following axiom on a model category:

\begin{defn} Given a class of maps $\C$ in $\M$, let $\C \otimes \M$ denote the class of maps $f \otimes id_X$ where $f\in \C$ and $X\in \M$. A model category is said to satisfy the \textit{monoid axiom} if every map in (Trivial-Cofibrations $\otimes \M)$-cell is a weak equivalence. 
\end{defn}

Let $A$ be any monoid and let $R$ be any commutative monoid. In \cite{SS00} and the follow-up paper \cite{hovey-monoidal}, it is proven that if $\M$ satisfies the monoid axiom and if the domains of $I$ (resp. $J$) are small relative to $(\M \otimes I$)-cell (resp. $(\M \otimes J$)-cell), then the categories of (left or right) $A$-modules and of $R$-algebras inherit model structures from $\M$. We will require the same smallness hypothesis in Section \ref{sec:main}. It is satisfied automatically if $\M$ is a combinatorial model category. 

In \cite{SS00} it is proven that it is sufficient to check the monoid axiom on the generating trivial cofibrations and that many model categories of interest satisfy the monoid axiom. We will conduct a similar program for the strong commutative monoid axiom in Section \ref{sec:examples} and in Appendix \ref{appendix-generators}.

\section{A model structure on commutative monoids}
\label{sec:main}

We are now ready to prove the commutative analog of the work summarized above. We first introduce the commutative analog to the monoid axiom.

\begin{defn} \label{defn:comm-monoid-ax}
A monoidal model category $\M$ is said to satisfy the \textit{commutative monoid axiom} if whenever $h$ is a trivial cofibration in $\M$ then $h^{\boxprod n}/\Sigma_n$ is a trivial cofibration in $\M$ for all $n>0$.
\end{defn}

Under this hypothesis, we state our main theorem:

\begin{theorem} \label{thm:commMonModel}
Let $\M$ be a cofibrantly generated monoidal model category satisfying the commutative monoid axiom and the monoid axiom, and assume that the domains of the generating maps $I$ (resp. $J$) are small relative to $(I\otimes \M)$-cell (resp. $(J\otimes \M$)-cell). Let $R$ be a commutative monoid in $\M$. Then the category $CAlg(R)$ of commutative $R$-algebras is a cofibrantly generated model category in which a map is a weak equivalence or fibration if and only if it is so in $\M$. In particular, when $R=S$ this gives a model structure on commutative monoids in $\M$.
\end{theorem}

It is clear from this description of $CAlg(R)$ that if $\M$ is combinatorial then $CAlg(R)$ is combinatorial (see \cite{beke-sheafifiable}, 2.3). 
Furthermore, if $\M$ is simplicial then $CAlg(R)$ is simplicial: it is cotensored over simplicial sets and the cotensor commutes with the forgetful functor (i.e. for $X$ in $\M$ and $K$ a simplicial set, $U(X^K)\cong U(X)^K$, since $X^K$ inherits its commutative monoid structure from $X$), the functor $X\mapsto X^K$ has a left adjoint for all $X$ and $K$ by the adjoint functor theorem, and this left adjoint provides the tensoring of $\M$ over simplicial sets. So $CAlg(R)$ is a simplicial category. To see that it is a simplicial model category, use that (trivial) fibrations are created in $\M$ and use the pullback formulation of the SM7 axiom. 

As the generating (trivial) cofibrations of $CAlg(R)$ are of the form $R\otimes \Sym(I)$ (resp. $R\otimes \Sym(J)$), these are cofibrant in $CAlg(R)$ if the generating (trivial) cofibrations of $\M$ are cofibrant and if $\M$ satisfies the commutative monoid axiom. Hence, the property that the domains of the generating cofibrations are cofibrant (sometimes called tractability) also passes from $\M$ to $CAlg(R)$.

\begin{proof}[Proof sketch]
We will focus first on the case where $R$ is the monoidal unit $S$, and discuss general $R$ at the end. As commutative $S$-algebras are simply commutative monoids, we denote the category of commutative monoids $\CMon(\M)$ rather than $CAlg(S)$. We will verify condition (1) of Lemma \ref{ss00-lemma2.3} for the monad coming from the $(\Sym, U)$ adjunction between $\M$ and $\CMon(\M)$. Note that just as in Lemma 2.3 of \cite{SS00}, limits in $\CMon(\M)$ are created in $\M$ and colimits in $\CMon(\M)$ exist because $U$ preserves filtered colimits. Thus, $\CMon(\M)$ is bicomplete. 
Let $J$ denote the generating trivial cofibrations of $\M$. We must prove that maps in $\Sym(J)$-cell are weak equivalences. Given a trivial cofibration $h:K \to L$ in $\M$, we form the following pushout square in $\CMon(\M)$:
\begin{align*}
\xymatrix{\Sym(K) \ar[r] \ar[d] \po & \Sym(L) \ar[d] \\ X \ar[r] & P}
\end{align*}
We must prove that the bottom map is in the class (Trivial-Cofibrations $\otimes \M)$-cell, so that the monoid axiom implies that transfinite compositions of such maps are weak equivalences in $\M$ (hence weak equivalences of commutative monoids). 

Of course, in $\CMon(\M)$, the pushout is simply the tensor product, so $P\cong X\otimes_{\Sym(K)}\Sym(L)$, but we will not make use of this fact. Following \cite{SS00}, we construct a filtration of the map of commutative monoids $X\to P$ as a composition $P_n \to P_{n+1}$ of maps formed by pushout diagrams in $\M$. Doing so requires the decomposition of $\Sym(K) = \coprod_n \Sym^n(K)$ where $\Sym^n(K) = K^{\otimes n}/\Sigma_n$. 

Thinking of $P$ as formal products of elements from $X$ and from $L$ with relations in $K$ leads to a consideration of $n$-dimensional cubes to build products of length $n$ from the letters $X,K,L$. Because the map $\Sym(K)\to X$ is adjoint to a map $K\to X$, we will in fact only need to consider $n$-dimensional cubes whose vertices are length $n$ words in the letters $K$ and $L$. Formally, for any subset $D$ of $[n] = \{1,2,\dots,n\}$ we obtain a vertex $C_1\otimes \dots \otimes C_n$ with $C_i = K$ if $i\notin D$ and $C_j = L$ if $j\in D$. The punctured cube is the cube with the vertex $L^{\otimes n}$ removed. The map $h^{\boxprod n}$ is the induced map from the colimit $Q_n$ of the punctured cube to $L^{\otimes n}$.

There is an action of $\Sigma_n$ on the cube which permutes the letters in the words (equivalently, which permutes the vertices in the cube in a way coherent with respect to the edges of the cube). Explicitly, the action is defined as follows. Any $\sigma \in \Sigma_n$ sends the vertex defined above to the vertex corresponding to $\sigma(D) \subset [n]$ using the action of $\Sigma_{|D|}$ on the $X$'s and $\Sigma_{n-|D|}$ on the $Y$'s. This action yields a $\Sigma_n$-action on $h^{\boxprod n}:Q_n \to L^{\otimes n}$, and in a moment we will pass to $\Sigma_n$-coinvariants.

We now show how to obtain $P_n$ (which in this analogy is to be thought of as formal products of length $n$) from the cubes we have just described. 
The steps in the filtration of $X\to P$ are formed by pushouts of the maps $id_X \otimes h^{\boxprod n}/\Sigma_n$:
\begin{align*}
\xymatrix{X \otimes Q_n/\Sigma_n \ar[r] \ar[d] \po & X\otimes L^{\otimes n}/\Sigma_n \ar[d] \\ P_{n-1} \ar[r] & P_n}
\end{align*}

The proof that the $P_n$ provide a filtration of $X\to P$ is delayed until Appendix \ref{appendix-proof}. Assuming the commutative monoid axiom, the maps $h^{\boxprod n}/\Sigma_n$ are trivial cofibrations. Thus, the map $X\to P$ is a transfinite composite of pushouts of maps in $\M \otimes \{$trivial cofibrations$\}$. Hence, by the monoid axiom, $X\to P$ is a weak equivalence. Similarly, for any transfinite composition $f$ of pushouts of maps of the form $\Sym(K)\to \Sym(L)$, we may realize $f$ as a transfinite composition of maps $X\to P$ of the form above. As a transfinite composition of transfinite compositions is still a transfinite composition, the monoid axiom applies again and proves $f$ is a weak equivalence. Lemma \ref{ss00-lemma2.3} now applies to produce the required model structure on commutative monoids.

To handle the case of commutative $R$-algebras, note that there is an equivalence of categories between $CAlg(R)$ and $(R\downarrow \CMon(\M))$, the category of commutative monoids under $R$. So we may apply the remark after Proposition 1.1.8 of \cite{hovey-book} to conclude that this is a model category with cofibrations, fibrations, and weak equivalences inherited from $\CMon(\M)$. Note that this is a different approach from the one provided in \cite{SS00} because we do not pass through $R$-modules en route to commutative $R$-algebras. That $CAlg(R)$ is cofibrantly generated follows from \cite{hirschhorn-under-categories}, where it is also shown that the generating cofibrations are given by the set $I_R$ of maps in $(R\downarrow \CMon(\M))$ where the map in $\CMon(\M)$ is in $I$. Under the equivalence of categories between $CAlg(R)$ and $(R\downarrow \CMon(\M))$, such maps are sent to maps in $R\otimes \Sym(I)$. We can similarly identify the generating trivial cofibrations as $R\otimes \Sym(J)$.

\end{proof}

\begin{remark} \label{remark:weak-cmon-axiom}
Notice that the proof in fact requires less than the full strength of the hypotheses. Rather than assuming the commutative monoid axiom and the monoid axiom separately, we could have assumed that transfinite compositions of pushouts of maps in $\{\M \otimes h^{\boxprod n}/\Sigma_n \;|\; h$ is a trivial cofibration$\}$ are contained in the weak equivalences. We will refer to this property as the weak commutative monoid axiom. Certain model categories discussed in Section \ref{sec:examples} only satisfy this axiom and not the commutative monoid axiom. However, for reasons which will become clear in Corollary \ref{cor:semi-on-comm-alg} we have chosen the commutative monoid axiom as the appropriate axiom for our applications.
\end{remark}

The full proof in Appendix \ref{appendix-proof} will in fact prove more than just the theorem. It will also prove the commutative analog to Lemma 6.2 of \cite{SS00}, from which one can deduce the proposition below regarding when cofibrations of commutative monoids forget to cofibrations in $\M$. It is well-known to experts that obtaining the correct behavior of cofibrations under the forgetful functor is subtle in the commutative setting. Indeed, this was the motivation behind the convenient model structures introduced in \cite{shipley-positive} and \cite{stolz-thesis}. In order to guarantee the desired behavior we must strengthen the commutative monoid axiom.

\begin{defn}
A monoidal model category $\M$ is said to satisfy the \textit{strong commutative monoid axiom} if whenever $h$ is a (trivial) cofibration in $\M$ then $h^{\boxprod n}/\Sigma_n$ is a (trivial) cofibration in $\M$ for all $n>0$. In particular, we are now assuming that cofibrations are closed under the operation $(-)^{\boxprod n}/\Sigma_n$.
\end{defn}

\begin{prop} \label{prop:cofibrations-forget}
Suppose $\M$ satisfies the strong commutative monoid axiom. Then for any commutative monoid $R$, a cofibration in $CAlg(R)$ with source cofibrant in $\M$ is a cofibration in $\M$.
\end{prop}

\begin{corollary} \label{cor:cofibrants-forget}
Suppose $\M$ satisfies the strong commutative monoid axiom and that $S$ is cofibrant in $\M$. Then any cofibrant commutative monoid is cofibrant in $\M$. If in addition $R$ is cofibrant in $\M$ then any cofibrant commutative $R$-algebra is cofibrant in $\M$.
\end{corollary}

See Appendix \ref{appendix-proof} for the proof of this proposition.

\begin{corollary}
Assume $S$ is cofibrant in $\M$ and that $\M$ satisfies the strong commutative monoid axiom. If $f$ is a cofibration between cofibrant objects then $\Sym(f)$ is a cofibration in $\M$. In particular, if $X$ is cofibrant in $\M$ then $\Sym(X)$ is cofibrant in $\M$.
\end{corollary}

\begin{proof}

Because the model structure on $\CMon(\M)$ is transferred from that of $\M$, the functor $\Sym(-)$ is left Quillen, and hence preserves cofibrations. So $\Sym(f)$ is a cofibration of commutative monoids because $f$ is a cofibration in $\M$. If the source $K$ of $f$ is cofibrant then the source of $\Sym(f)$ is a cofibrant commutative monoid, by applying $\Sym(-)$ to the cofibration $\emptyset \hookrightarrow K$. By Corollary \ref{cor:cofibrants-forget}, the source of $\Sym(f)$ is cofibrant in $\M$. By Proposition \ref{prop:cofibrations-forget}, $\Sym(f)$ is a cofibration in $\M$.
\end{proof}

Recall that the point of positive model structures on diagram spectra (e.g. symmetric spectra or orthogonal spectra) is to break the cofibrancy of $S$ and so avoid Lewis's obstruction \cite{lewis} to having a model structure on commutative ring spectra. Thus, these corollaries do not apply to positive model categories of spectra. In \cite{shipley-positive}, a variant on the positive model structure is introduced in which cofibrant commutative ring spectra are cofibrant as spectra. This model structure was known in that paper as the convenient model structure, and later as the positive flat model structure. We do not know how to obtain this `convenient' property for general model categories. We suspect it has something to do with forcing the cofibrations to contain the monomorphisms.

The proof of Theorem \ref{thm:commMonModel} makes clear precisely where the monoid axiom is being used, and hence why the smallness hypotheses are needed. If $\M$ does not satisfy the monoid axiom, then we can make this step work by assuming $X$ is a cofibrant commutative monoid (note that $R$ does not need to be cofibrant). In this case, \cite{hovey-monoidal} and \cite{spitzweck-thesis} make it clear that a semi-model structure can be obtained. The following corollary is needed for \cite{white-localization}, and the notion of a semi-model category is defined directly afterwards. The proof of this corollary is delayed until after Proposition \ref{prop:build-semi}.

\begin{corollary} \label{cor:semi-on-comm-alg}
Let $\M$ be a cofibrantly generated monoidal model category satisfying the commutative monoid axiom, and assume that the domains of the generating maps $I$ (resp. $J$) are small relative to $(I\otimes \M)$-cell (resp. $(J\otimes \M$)-cell). 
Then for any commutative monoid $R$, the category of commutative $R$-algebras is a cofibrantly generated semi-model category in which a map is a weak equivalence or fibration if and only if it is so in $\M$.
\end{corollary}

A semi-model category satisfies all the axioms of a model category except that the lifting of trivial cofibrations against fibrations is only true for trivial cofibrations with cofibrant domains, and only maps $f$ with cofibrant domain are guaranteed to factor into a trivial cofibration followed by a fibration. In particular, if all objects are cofibrant then a semi-model structure is the same as a model structure. Formally, we define (\cite{spitzweck-thesis}, Definition 1):

\begin{defn} \label{defn:semi-model}
A \textit{semi-model category} is a bicomplete category $\D$, an adjunction $F:\M \leftrightarrows \D:U$ where $\M$ is a model category, and subcategories of weak equivalences, fibrations, and cofibrations in $\D$ satisfying the following axioms:
\begin{enumerate}
\item $U$ preserves fibrations and trivial fibrations.
\item $\D$ satisfies the two out of three axiom and the retract axiom.
\item Every map in $\D$ can be functorially factored into a cofibration followed by a trivial fibration. Every map in $\D$ whose domain is cofibrant in $\D$ can be functorially factored into a trivial cofibration followed by a fibration.
\item Cofibrations in $\D$ have the left lifting property with respect to trivial fibrations. Trivial cofibrations in $\D$ whose domain is cofibrant in $\D$ have the left lifting property with respect to fibrations.
\item The initial object in $\D$ is cofibrant in $\D$.
\item Fibrations and trivial fibrations are closed under pullback.
\end{enumerate}
$\D$ is said to be \textit{cofibrantly generated} if there are sets of morphisms $I'$ and $J'$ in $\D$ such that $I'$-inj is the class of trivial fibrations and $J'$-inj the class of fibrations in $\D$, if the domains of $I'$ are small relative to $I'$-cell, and if the domains of $J'$ are small relative to maps in $J'$-cell with cofibrant domain.
\end{defn}

Spitzweck \cite{spitzweck-thesis} referred to this as a $J$-semi model category. It has more structure than the semi-model categories of \cite{fresse-book} (which Spitzweck called $(I,J)$-semi model categories), but less structure than a $J$-semi model category \textit{over $\M$} (where cofibrancy in $\D$ is weakened to cofibrancy in $\M$). Semi-model structures as defined above often arises in practice, as the following proposition demonstrates:

\begin{proposition} \label{prop:build-semi}
Let $\M$ be a cofibrantly generated monoidal model category, $T$ a monad in $\M$ such that $T$ commutes with filtered colimits, and suppose:
\begin{enumerate}
\item[a] The initial $T$-algebra is cofibrant, i.e. has the left lifting property against trivial fibrations of $T$-algebras.
\item[b] The domains of the generating maps $TI$ (resp. $TJ$) are small relative to $TI$-cell (resp. maps in $TJ$-cell with cofibrant domain).
\item[c] Every map in $TJ$-cell whose domain is cofibrant as a $T$-algebra is a weak equivalence in $\M$.
\end{enumerate}
Then $T$-alg inherits a cofibrantly generated semi-model structure from $\M$ where weak equivalences and fibrations are created and reflected by $U:T$-alg$\to \M$, and where cofibrations are defined by lifting.
\end{proposition}

\begin{proof}
We check Definition \ref{defn:semi-model} item by item. $T$-alg is complete because limits are computed in $\M$. $T$-alg is cocomplete because $T$ commutes with filtered colimits, just as in Lemma 2.3 of \cite{SS00}. $U$ preserves fibrations and trivial fibrations by the definition of these classes. The two out of three axiom and the retract axiom for weak equivalences and fibrations follows from the same axioms in $\M$. The retract axiom for cofibrations follows because $TI$-cof is closed under retracts. We thus have (1)-(2) of Definition \ref{defn:semi-model}.

By assumption (b), we can use the small object argument to factor any map $f$ as $p\circ i$ or as $p' \circ i'$, where $i$ is in $TI$-cell, $i'$ is in $TJ$-cell, $p$ is in $TI$-inj, and $p'$ is in $TJ$-inj. By adjointness (as in Lemma 2.3 of \cite{SS00}), $p$ is a trivial fibration, $p'$ is a fibration, and $i$ and $i'$ are cofibrations. The $p\circ i$ factorization is half of (3). If $f$ has cofibrant domain then (c) implies $i'$ is a weak equivalence, hence a trivial cofibration, completing (3).

For the first half of (4), note that cofibrations lift against trivial fibrations by definition. Let $f$ be a trivial cofibration with cofibrant domain. Factor $f=p'\circ i'$ into an element of $TJ$-cell (just shown to be a weak equivalence) followed by a fibration. By the two out of three axiom, $p'$ is a trivial fibration, so $f$ lifts against $p'$, so $f$ is a retract of $i'$ by the retract argument. It follows that $f$ has the left lifting property with respect to fibrations, completing (4).

Lastly, (5) holds by (a) and (6) holds because limits are computed in $\M$, where fibrations and trivial fibrations are preserved under pullback.
\end{proof}

This proposition is related to Theorem 2 in \cite{spitzweck-thesis}, but has fewer conditions and produces a semi-model category rather than a semi-model category over $\M$. We are now ready to prove the Corollary:

\begin{proof}[Proof of Corollary \ref{cor:semi-on-comm-alg}]
We use Proposition \ref{prop:build-semi}, where $T$ is $\Sym$. First, (a) is true because the initial object is $S = \Sym(\emptyset)$, and so is cofibrant in $\CMon(\M)$ by adjointness, since $\emptyset$ is cofibrant in $\M$. Next, (b) is true by the smallness assumption together with the filtration of $\Sym$ given in the proof of Theorem \ref{thm:commMonModel}. We are therefore reduced to checking (c). We begin with the case where $R=S$, so that we are building a semi-model structure on $\CMon(\M)$. The proof of Theorem \ref{thm:commMonModel} does not use the monoid axiom until we have already proven that the pushout of commutative monoids
\begin{align*}
\xymatrix{\Sym(K)\ar[r] \ar[d] & \Sym(L)\ar[d] \\ X\ar[r] & P}
\end{align*}
can be factored into $X=P_0\to P_1\to \dots \to P$ where each $P_{n-1} \to P_n$ is a pushout of $X\otimes f^{\boxprod n}/\Sigma_n$. By the commutative monoid axiom, $f^{\boxprod n}/\Sigma_n$ is a trivial cofibration. Without the monoid axiom it is not clear how to proceed unless $X$ is cofibrant. Every map in $TJ$-cell whose domain is cofibrant is a transfinite composition of maps of the form $X\to P$ above, where $X$ is cofibrant, so we may assume $X$ is cofibrant when verifying Proposition \ref{prop:build-semi} (c). In this case, the map $X\otimes f^{\boxprod n}/\Sigma_n$ has the form $(\emptyset \hookrightarrow X)\boxprod f^{\boxprod n}/\Sigma_n$ and hence is a trivial cofibration by the pushout product axiom. Thus, the pushout $P_{n-1}\to P_n$ must also be a trivial cofibration, and the composite $X\to P$ is a composite of trivial cofibrations and hence a trivial cofibration. 

For the case of a general commutative monoid $R$, observe that $CAlg(R)=R\downarrow \CMon(\M)$. For undercategories, one usually defines a map $f:X\to Y$ (i.e. a commutative triangle under $R$) to be a cofibration, fibration, or weak equivalence precisely when $f$ is a cofibration, fibration, or weak equivalence in $\CMon(\M)$. The semi-model structure we seek on $CAlg(R)$ has the same weak equivalences and fibrations (hence the same cofibrations) as these three classes of maps. Theorem 2.7 of \cite{hirschhorn-under-categories} demonstrates that these classes can be transferred from $\CMon(\M)$ along the adjunction $R\otimes -:\CMon(\M)\leftrightarrows CAlg(R):U$.
Since the semi-model structure on $\CMon(\M)$ is itself transferred from $\M$, this means the semi-model structure we seek on $CAlg(R)$ can be viewed as being transferred along the adjunction $R\otimes \Sym(-):\M \leftrightarrows CAlg(R):U$.

In order to check that this transfer defines a semi-model structure, we use Proposition \ref{prop:build-semi} with $T = R\otimes \Sym(-)$. Hypothesis (a) is true by adjunction, because the initial object $R$ is $T(\emptyset)$ and $\emptyset$ is cofibrant in $\M$. Note that this does not require $R$ to be cofibrant in $\M$, only in $CAlg(R)$. Hypothesis (b) is true by our smallness assumption and the filtration above. Hypothesis (c) is true by an argument analogous to the $R=S$ case above, but now where everything in sight is viewed in $R\downarrow \CMon(\M)$. Namely, we use Proposition 2.4 in \cite{hirschhorn-under-categories} to see that the pushout of $X\gets R\otimes \Sym(K)\to R\otimes \Sym(L)$ in $CAlg(R)$ is simply $P$. We analyze this pushout in the case where $X$ is cofibrant in $\CMon(\M)$ (equivalently, $R\to X$ is cofibrant in $CAlg(R)$), we write the filtration maps $P_{n-1}\to P_n$ as pushouts of trivial cofibrations, and we conclude that $X\to P$ is a trivial cofibration. Since the maps in $TJ$-cell whose domains are cofibrant are transfinite compositions of such maps $X\to P$, we conclude (c) holds and hence that $CAlg(R)$ has a transferred semi-model structure.
\end{proof}

Observe that if one wishes to obtain on $CAlg(R)$ a \textit{semi-model structure over $\M$} in the terminology of \cite{spitzweck-thesis} then one must also assume that $S$ and $R$ are cofibrant so that the initial object in $CAlg(R)$ forgets to a cofibrant object in $\M$. Note that the proof given here is fundamentally different from Theorem 3.3 in \cite{hovey-monoidal} (which required $R$ to be cofibrant), because we do not pass through the category of $R$-modules, and so we do not need to prove $R$-mod is a monoidal model category.

As the filtration given in Appendix \ref{appendix-proof} is related to Harper's filtration for general operads from \cite{harper-operads}, we pause for a moment to compare these two approaches.

\begin{remark} \label{remark-enveloping-computations}
Harper's general machinery describes the map $P_{n-1}\to P_n$ as a pushout
\begin{align*}
\xymatrix{Com_X[n] \otimes_{\Sigma_n} Q_n \ar[r] \ar[d] \po & Com_X[n] \otimes_{\Sigma_n} L^{\otimes n} \ar[d] \\ P_{n-1} \ar[r] & P_n}
\end{align*}

where $Com_X$ is the enveloping operad. One may use Proposition 7.6 in \cite{harper-operads} to write $Com_X[n] = X$ with the trivial $\Sigma_n$ action. Thus, $P_{n-1}\to P_n$ can be written as the pushout of $X\otimes f^{\boxprod n}/\Sigma_n$ and Harper's filtration makes it clear that the commutative monoid axiom is precisely the right hypothesis.

In a similar way, $Ass_X[n] = X^{\otimes n+1} \cdot \Sigma_n$, i.e. the coproduct of $n!$ copies of $X^{\otimes n+1}$ with the free $\Sigma_n$ action. So in that case the $(-\cdot \Sigma_n) \otimes_{\Sigma_n}(-)$ provides a cancellation and Harper's filtration reduces to a pushout of $X^{\otimes n+1}\otimes f^{\boxprod n}$. We see immediately why the monoid axiom is necessary.

Finally, one could realize commutative $R$-algebras as algebras over the operad $Com_R$ and in this case Harper's filtration would be a pushout of a map of the form $(Com_R)_A[n] \otimes_{\Sigma_n} f^{\boxprod n}$ where $A$ is a commutative $R$-algebra. In this case, the formula in Proposition 7.6 yields $(Com_R)_A[n] = R\otimes A$ and so the maps $P_{n-1}\to P_n$ are pushouts of $(R\otimes A) \otimes f^{\boxprod n}/\Sigma_n$. In this way we see that in the presence of the commutative monoid axiom but in the absence of the the monoid axiom we need both $R$ and $A$ to be cofibrant in order to ensure that this map is a trivial cofibration, i.e. to obtain on $CAlg(R)$ a semi-model structure over $\M$. This is the commutative analog of Theorem 3.3 in \cite{hovey-monoidal}, in which cofibrancy of $R$ was required to achieve a semi-model structure on $R$-algebras. There the formula $(Ass_R)_A[n] = R\otimes A \cdot \Sigma_n$ means that the relevant pushout takes the form $R\otimes A \otimes f^{\boxprod n}$ and this makes clear why both $R$ and $A$ must be cofibrant in the absence of the monoid axiom.

\end{remark}

We conclude this section with a remark comparing our approach and results with the approach outlined by Lurie in \cite{DAG3}, in which he proved:

\begin{theorem} \label{thm:Lurie}
Let $\M$ be a left proper, combinatorial, tractable, monoidal model category satisfying the monoid axiom and with a cofibrant unit. Assume further that 

(*) If $h$ is a cofibration then $h^{\boxprod n}$ is a cofibration in the projective model structure on $\M^{\Sigma_n}$ for all $n$. Such maps $h$ are called power cofibrations.

Then $\CMon(\M)$ has a model category structure with weak equivalences and fibrations inherited from $\M$.
\end{theorem}

The difference between this result and Theorem \ref{thm:commMonModel} is that in Theorem \ref{thm:commMonModel} we do not require $\M$ to be left proper, we do not require the unit to be cofibrant, we do not require the model structure to be tractable, we weaken combinatoriality to a much lesser smallness hypothesis, and we weaken (*) to the commutative monoid axiom. We have also discussed how to remove the monoid axiom. Note that Lurie also assumes $\M$ is simplicial, but never uses this assumption. The assumption that the unit is cofibrant is part of what Lurie requires of a monoidal model category. However, the unit is not cofibrant in the positive and positive flat model structures on categories of spectra. For this reason, Theorem \ref{thm:Lurie} cannot apply to such examples as stated, but elements of the proof have been made to apply to the positive flat stable model structure in \cite{luis}. 

We refer to condition (*) as Lurie's hypothesis. It implies the strong commutative monoid axiom as shown in Lemma 4.3.28 of \cite{DAG3}. The key observation is that $(-)/\Sigma_n:\M^{\Sigma_n}\to \M$ is the left adjoint of a Quillen pair where the right adjoint is the constant diagram functor (i.e. endows an object with the trivial $\Sigma_n$ action). Thus, if (*) is satisfied and we apply this map to the projective cofibration $f^{\boxprod n}$ we obtain the strong commutative monoid axiom. However, (*) assumes strictly more than the strong commutative monoid axiom, as evidenced in Section \ref{sec:examples} where we show that simplicial sets and topological spaces satisfy the latter but not the former.

Note that Lurie's Proposition 4.3.21 is slightly more general than what we've stated above in that it only requires that there is \emph{some} combinatorial model structure $\M_V$ on the relative category $\M$, and that $\M_V$ has cofibrations $V$ generated by cofibrations between cofibrant objects and satisfying (*). In this case $\M$ is said to be \textit{freely powered} by $V$. We could also do our work in that level of generality, but choose not to because it seems unnatural to place a hypothesis on a model category which references the existence of some other model category. The point is that this extra generality does not buy us anything because $\M$ and $\M_V$ will be Quillen equivalent by Lurie's Remark 4.3.20.

Lurie does not prove that it is sufficient to check hypothesis (*) on the generating (trivial) cofibrations, but this has been done in \cite{luis}. 

\section{Additional Results} \label{sec:additional-results}

\subsection{Functoriality and Homotopy Invariance}

We turn now to the question of whether or not the passage from $R$ to $CAlg(R)$ is functorial and has good homotopy theoretic properties. Following \cite{SS00}, we provide a condition so that the homotopy theory of commutative $R$-algebras only depends on the weak equivalence type of $R$. Recall that a monoidal model category $\M$ is said to satisfy the property that \textit{cofibrant objects are flat} if for all cofibrant $X$ and all weak equivalences $f$, the map $X\otimes f$ is a weak equivalence. This property can be viewed as a global version of the unit axiom (which is the same statement restricted to the cofibrant replacement map $f:QS\to S$).

\begin{theorem} Suppose $\M$ satisfies the conditions of Theorem \ref{thm:commMonModel}. Then:

\begin{enumerate}

\item The passage from $R$ to $CAlg(R)$ is functorial: given a ring homomorphism $f:R\to T$, restriction and extension of scalars provides a Quillen adjunction between $CAlg(R)$ and $CAlg(T)$.

\item Suppose that cofibrant objects are flat in $CAlg(R)$, i.e. for any cofibrant commutative $R$-algebra $N$, the functor $N\otimes_R -$ preserves weak equivalences of commutative $R$-algebras. Let $f:R\to T$ be a weak equivalence of commutative monoids. Then $f$ induces a Quillen equivalence $CAlg(R) \simeq CAlg(T)$.
\end{enumerate}
\end{theorem}

\begin{proof} Let $f:R\to T$ be a ring homomorphism.

\begin{enumerate}

\item The map $f$ makes $T$ into an $R$-module, and provides the extension of scalars functor from $CAlg(R)$ to $CAlg(T)$, i.e. $N \cong R\otimes_R N \to T\otimes_R N$. Because weak equivalences and fibrations are defined in the underlying category, the right adjoint restriction functor preserves (trivial) fibrations. So they form a Quillen pair and the extension functor preserves (trivial) cofibrations. 

\item To check that extension and restriction form a Quillen equivalence in this case, we use Corollary 1.3.16(c) of \cite{hovey-book}. First, note that restriction reflects weak equivalences between fibrant objects because the weak equivalences and fibrations in these two categories are the same. Next, suppose $N$ is a cofibrant commutative $R$-algebra. The map $N \cong R\otimes_R N \to T\otimes_R N$ is a weak equivalence because cofibrant objects are flat. Thus, restriction and extension of scalars form a Quillen equivalence.
\end{enumerate}
\end{proof}

An alternative approach for (2) which avoids the need for cofibrant $R$-modules to be flat is suggested by Theorem 2.4 of \cite{hovey-monoidal} in the non-commutative case. The cost is a collection of cofibrancy hypotheses on the objects in question. Via Remark \ref{remark-enveloping-computations} we may view the generating cofibrations of $CAlg(R)$ as $R\otimes \Sym(I)$ where $I$ is the set of generating cofibrations for $\M$.

\begin{theorem}
Suppose $\M$ has a cofibrant unit, satisfies the strong commutative monoid axiom, and that the domains of the generating cofibrations are cofibrant. Suppose $R$ and $T$ are commutative monoids which are cofibrant in $\M$ and suppose $f:R\to T$ is a weak equivalence. Then extension and restriction of scalars is a Quillen equivalence between $CAlg(R)$ and $CAlg(T)$.
\end{theorem}

\begin{proof}

We follow the model of Hovey's proof in \cite{hovey-monoidal}. All that must be shown is that for all cofibrant $R$-modules $M$, $M\to M\otimes_R T$ is a weak equivalence. Because $M$ is cofibrant we may write $M = \colim M_\alpha$ where $M_0=0$ and $M_\alpha \to M_{\alpha+1}$ is a pushout of a map in $R\otimes \Sym(I)$. For concreteness we will let $K\to L$ denote the map in $I$ which is used in this pushout.

We show by transfinite induction that $M_\alpha \to M_\alpha \otimes_R T$ is a weak equivalence for all $\alpha$. The base case is trivial because $M_0=0$. For the successor case, apply the left adjoint $-\otimes_R T$ to the pushout square defining $M_\alpha \to M_{\alpha+1}$ and the result will again be a pushout square. There is also a map from the former pushout square to the latter, induced by the adjunction. We will apply the Cube Lemma (Lemma 5.2.6 in \cite{hovey-book}) to the resulting cube.
\begin{align*}
\xymatrix{R\otimes \Sym(K) \ar[r] \ar[d] & R\otimes \Sym(L)\ar[d] &  & 
T\otimes \Sym(K) \ar[r] \ar[d] & T\otimes \Sym(L) \ar[d] \\ 
M_\alpha \ar[r] & M_{\alpha+1} & & M_\alpha \otimes_R T \ar[r] & M_{\alpha+1} \otimes_R T}
\end{align*}

Here we have canceled $R\otimes_R (-)$ terms in the right-hand square. Because $\M$ has the commutative monoid axiom and a cofibrant unit, the cofibrancy of $K$ and $L$ implies the cofibrancy of $\Sym(K)$ and $\Sym(L)$ in $\M$, by Corollary \ref{cor:cofibrants-forget}. 
Thus, by Lemma 1.1.12 in \cite{hovey-book}, smashing with these objects preserves weak equivalences between cofibrant objects, so when we apply this to the weak equivalence $R\to T$, we see that the maps $\otimes \Sym(K)\to T\otimes \Sym(K)$ and $R\otimes \Sym(L) \to T\otimes \Sym(L)$ above are weak equivalences. Similarly, the map $\Sym(K)\to \Sym(L)$ is a cofibration and so because $R$ and $T$ are cofibrant the horizontal maps across the top are cofibrations (and hence the bottom horizontals as well, because they are pushouts of cofibrations). 

Because all maps $M_{\alpha} \to M_{\alpha+1}$ are cofibrations and because $M_0$ is cofibrant, all $M_\alpha$ are cofibrant. Because extension of scalars is left Quillen, the objects in the second square are cofibrant. The inductive hypothesis tells us that the map on the lower left corner is a weak equivalence. The Cube Lemma then guarantees us that the map on the lower right corner is a weak equivalence.

For the limit ordinal case, assume that $M_\alpha \to M_\alpha \otimes_R T$ is a weak equivalence for all $\alpha < \lambda$. Then we have a map of sequences
\begin{align*}
\xymatrix{M_0 \ar[r] \ar[d] & M_1 \ar[r] \ar[d] & \dots \ar[r] & M_\alpha \ar[r] \ar[d] & \dots \\
M_0 \otimes_R T \ar[r] & M_1 \otimes_R T \ar[r] & \dots \ar[r] & M_\alpha \otimes_R T \ar[r] & \dots}
\end{align*}

where all vertical maps are weak equivalences and all all horizontal maps are cofibrations of cofibrant objects. So Proposition 18.4.1 in \cite{hirschhorn} proves the colimit map $M_\lambda \to M_\lambda \otimes_R T$ is a weak equivalence as well.

\end{proof}

Hovey provides a counterexample which demonstrates that for non-cofibrant $R$ and $T$, and without the hypothesis that cofibrant $R$-modules are flat, it is not true that $R\simeq T$ induces a Quillen equivalence of categories of modules. 

We do not know whether or not Hovey's example can be generalized to the case of algebras rather than modules. We do know that the spaces considered in Hovey's example cannot provide such a counterexample for the question of Quillen equivalence between $CAlg(R)$ and $CAlg(T)$, because commutative monoids in $Top$ are generalized Eilenberg-Mac Lane spaces (as discussed in Example \ref{example:top-fails-rectification}).

The author does not know whether or not it is possible to prove homotopy invariance of $CAlg(R)$ without the hypothesis that cofibrant objects are flat and without having to assume the objects $R$ and $T$ are cofibrant. Note that Corollary 2.4 of \cite{BM07} does not apply here because the operads $Com$, $Com_R$, and $Com_T$ are not $\Sigma$-cofibrant. 

\begin{remark}
The results in this section also hold in the absence of the monoid axiom. By Corollary \ref{cor:semi-on-comm-alg}, categories of commutative algebras form semi-model categories and the output of the theorem is a Quillen equivalence of semi-model categories. To see this one need only note that the monoid axiom is not used in the proof, and that the semi-model category analog of 1.3.16 in \cite{hovey-book} can be found in Section 12.1.8 of \cite{fresse-book}.
\end{remark}

\subsection{Rectification} \label{subsec:rectification}

We turn next to the question of rectification. As discussed in \cite{spitzweck-thesis}, categories of algebras over cofibrant operads inherit model structures whenever the monoid axiom is satisfied. Thus, $E_\infty$-algebras in $\M$ will always inherit a model structure in our set-up. There is a weak equivalence $\phi:E_\infty \to Com$, so it is natural to ask whether or not the pair $(\phi^*,\phi_!)$ forms a Quillen equivalence between $E_\infty$-algebras and $Com$-algebras. If there is, then \textit{rectification} is said to occur.

Observe that rectification does not come for free, even for very nice model categories $\M$, as the following counterexample demonstrates:

\begin{example} \label{example:top-fails-rectification}
Let $\M$ be simplicial sets or topological spaces. We will see in the next section that $\M$ satisfies the strong commutative monoid axiom. The monoid axiom and requisite smallness were verified in \cite{SS00} for simplicial sets, in \cite{hovey-monoidal} for compactly generated spaces, and in \cite{white-topological} for $k$-spaces. Thus, commutative monoids inherit a model structure.

For topological spaces the path connected commutative monoids are weakly equivalent to generalized Eilenberg-Mac Lane spaces, i.e. products of Eilenberg-Mac Lane spaces. 

The existence of spaces like $QS = \Omega^\infty \Sigma^\infty S^0$, which has an $E_\infty$-algebra structure but is not a generalized Eilenberg-Mac Lane space, demonstrates that rectification between $E_\infty$ and $Com$ fails for spaces. 
\end{example}

The rectification results of \cite{BM07} are phrased so as to apply for very general model categories $\M$, including simplicial sets. However, these results do not apply to the example above because $Com$ is not a $\Sigma$-cofibrant operad. If $\M$ satisfies Harper's hypothesis that all symmetric sequences are projectively cofibrant (e.g. if $\M = Ch(k)$ for $k$ a $\Q$-algebra), then $Com$ is $\Sigma$-cofibrant and so rectification holds.

The key property possessed by good monoidal categories of spectra is

(**) For all cofibrant $X$, the map $(E\Sigma_n)_+ \wedge_{\Sigma_n} X^{\wedge n} \to X^{\wedge n}/\Sigma_n$ is a weak equivalence. 

This property is certainly related to the commutative monoid axiom, and it is often used to verify the commutative monoid axiom for positive model structures on symmetric and orthogonal spectra. However, the example above demonstrates that this property is not necessary for strictly commutative monoids to inherit a model structure, and that it cannot be deduced from the commutative monoid axiom. We now record the correct analogue of this property (**) in general model categories. We will assume $\M$ is a $\D$-model category in the sense of Definition 4.2.18 in \cite{hovey-book}, and this allows operads valued in $\D$ to act in $\M$.

\begin{defn}
Let $\M$ be a monoidal model category which is a $\D$-model category. View the unit $S$ of $\D$ as an object in $\D^{\Sigma_n}$ with the trivial $\Sigma_n$ action. Let $q:Q_{\Sigma_n}S\to S$ be cofibrant replacement in the projective model structure on $\D^{\Sigma_n}$. Then $\M$ is said to satisfy the \textit{rectification axiom} with respect to operads valued in $\D$ if for all cofibrant $X$ in $\M$, the natural map $Q_{\Sigma_n} S \otimes_{\Sigma_n} X^{\otimes n} \to X^{\otimes n}/\Sigma_n$ is a weak equivalence.
\end{defn}

A similar property to the rectification axiom, requiring certain homotopy orbits to be weakly equivalent to orbits, appears in the axiomatization of good model structures of spectra given by \cite{goerss-hopkins-moduli-problems}. However, in \cite{goerss-hopkins-moduli-problems}, this condition is equivalent to the condition that all simplicial operads are admissible, and as we have seen that will not be true for general model categories. 

The key consequence of the rectification axiom is that rectification will occur between commutative monoids and algebras over a cofibrant replacement $QCom$ of the $Com$ operad (see Theorem \ref{thm:rect-ax-implies-rect} below). An example of such rectification is the Quillen equivalence between commutative ring spectra and $E_\infty$-algebras in good monoidal model categories of spectra, where $QCom$ can be taken to be the Fulton-Macpherson operad. In general, we work in the setting of $\D$-model categories where $\D$ is a monoidal model category (see \cite{hovey-book}, 4.2.6). A $\D$-operad $O$ has $O(n)$ an object in $\D$ for all $n$. These $\D$-operads have algebras in $\M$, just as simplicial operads have algebras in categories of spectra. There is a semi-model structure on the category of $\D$-operads transferred from the projective model structure on symmetric sequences in $\D$ (\cite{fresse-book}, 12.2.A). We define $QCom$ to be a cofibrant replacement of $Com$ in this semi-model structure. This semi-model structure is often a full model structure if $\D$ satisfies stronger conditions (\cite{BM03} Theorem 3.1), but a semi-model structure suffices for our needs.

\begin{theorem} \label{thm:rect-ax-implies-rect}
Suppose that
\begin{enumerate}
\item $\D$ is a monoidal model category whose unit is cofibrant.
\item $\M$ is a monoidal $\D$-model category satisfying the strong commutative monoid axiom and the monoid axiom.
\item The domains of the generating cofibrations of $\M$ are cofibrant and satisfy the requisite smallness hypotheses so that $Com$-alg and $QCom$-alg may inherit transferred model structures.
\item Either that the unit $S$ of $\M$ is cofibrant, or that $\M$ is the positive (flat) model structure on symmetric or orthgonal spectra.
\item $\M$ satisfies the rectification axiom. 
\end{enumerate} 
Then the cofibrant replacement morphism $\phi:QCom \to Com$ induces a Quillen equivalence between $Com$-alg and $QCom$-alg.
\end{theorem}

Condition (4) above guarantees that cofibrations of commutative monoids with cofibrant source forget to cofibrations in $\M$ (by Proposition \ref{prop:cofibrations-forget} in the former case and Proposition 4.1 in \cite{shipley-positive} in the latter). Similarly, $QCom$-algebras with cofibrant source forget to cofibrations in $\M$. For the setting where $S$ is cofibrant, this follows from Theorem 12.1.4 in \cite{fresse-book} and Proposition 4.3 in \cite{BM03}. For the setting of spectra, we will prove this as part of Corollary \ref{cor:rect-for-spectra} below.

We will see that this theorem implies rectification results for positive (flat) monoidal categories of spectra (where $\D$ is simplicial sets or topological spaces) and for chain complexes. The case of $G$-equivariant spectra (where $\D$ is $G$-equivariant spaces) is more subtle. In that setting, the model structure on $G$-operads (Theorem 3.1 in \cite{BM03}, Theorem 12.2.A in \cite{fresse-book}) is the wrong model structure to correctly encode the mixing of group actions between $G$ and the operad symmetric group actions. As a consequence, $QCom$ is not an $N_\infty$-operad in the sense of \cite{blumberg-hill}, and the rectification axiom is not satisfied. However, in current joint work with Javier Guti\'{e}rrez, the author shows that there is another model structure on $G$-operads in which the cofibrant replacement of $Com$ is a cofibrant $N_\infty$-operad and in which rectification occurs. The rectification axiom for that setting states that for all cofibrant $G$-spectra $X$, the natural map $(E_G \Sigma_n)_+ \wedge_{\Sigma_n} X^{\wedge n} \to X^{\wedge n}/\Sigma_n$ is a weak equivalence of $G$-spectra, where $E_G \Sigma_n$ is the total space of the universal $G$-equivariant principal $\Sigma_n$-bundle. We refer to the existence of this weak equivalence for all cofibrant $X$ as the \textit{equivariant analogue of the rectification axiom}.

Before proving the theorem above, we record a lemma regarding the behavior of weak equivalences under coproduct.

\begin{lemma} \label{lemma:coprod-of-weak-equivs}
Arbitrary weak equivalences between cofibrant objects are closed under coproduct.
\end{lemma}

\begin{proof}
Suppose $\{f_\alpha:A_\alpha \to B_\alpha\}_{\alpha \in S}$ is a set of weak equivalences between cofibrant objects. Form the model category $\prod_{\alpha \in S} \M$ with weak equivalences, cofibrations, and fibrations defined from $\M$. Consider the functor $F$ from this model category to $\M$ which takes $(A_\alpha)$ to $\coprod_{\alpha \in S}A_\alpha$. This functor takes trivial cofibrations between cofibrant objects to trivial cofibrations, since the coproduct of any collection of trivial cofibrations in $\M$ is a trivial cofibration. Hence, by Ken Brown's lemma (Lemma 1.1.12 in \cite{hovey-book}) this functor takes weak equivalences between cofibrant objects to weak equivalences. The map $(f_\alpha)$ in $\prod_{\alpha \in S} \M$ is a weak equivalence between cofibrant objects, so $\coprod f_\alpha$ is a weak equivalence in $\M$.
\end{proof}

\begin{proof}[Proof of Theorem \ref{thm:rect-ax-implies-rect}]
First, the restriction functor $\phi^*$ preserves limits and so is a right adjoint, as can be seen from Theorem 12.5.A in \cite{fresse-book}. The left adjoint is denoted by $\phi_!$. Next, $\phi^*$ commutes with the forgetful functor to $\M$ and so preserves (trivial) fibrations of algebras since these are created in $\M$. Thus, $\phi^*$ is a right Quillen functor. 

To show that this Quillen pair is a Quillen equivalence, we will use that $\phi^*$ reflects weak equivalences. This reduces us to proving that for all cofibrant $X$ in $QCom$-alg, the adjunction unit map $\eta:X\to \phi^* \phi_!(X)$ is a weak equivalence. We carry this out first for the case where $X$ has the form $QCom(A)$ for some cofibrant $A$. In this case, $\phi^* \phi_!(X) = \Sym(A)$ and the map $\eta: QCom(A)\to \Sym(A)$ is induced by $\phi$. This map $\eta$ is a coproduct of maps of the form $\eta_n:QCom(n)\otimes_{\Sigma_n} A^{\otimes n} \to A^{\otimes n}/\Sigma_n$. Since $QCom$ is a cofibrant replacement for $Com$, it is in particular a $\Sigma$-cofibrant replacement by Proposition 4.3 in \cite{BM03} applied to the category of operads in $\D$ (this is where we need the hypothesis that the unit of $\D$ is cofibrant) and so $QCom(n)\cong Q_{\Sigma_n}S$. Thus, $\eta_n$ is a weak equivalence for all $n$ by the rectification axiom. 

Because $Q_{\Sigma_n}S$ is $\Sigma_n$-cofibrant and $A^{\otimes n}$ is cofibrant in $\M$, the domain of $\eta_n$ is a cofibrant object in $\M$ (using Lemma 2.5.2 in \cite{BM06}). The codomain of $\eta_n$ is cofibrant by the commutative monoid axiom on $\M$. Thus, Lemma \ref{lemma:coprod-of-weak-equivs} implies $\eta$ is a weak equivalence as required.

Since every cofibrant object in $QCom$-alg is a retract of a cellular object, it suffices to prove that $\eta_X:X\to \phi^*\phi_!X$ is a weak equivalence for cellular $X$, i.e. those built as transfinite compositions of pushouts of generating cofibrations of $QCom$-alg. Thus, $X\cong \colim X_\alpha$, where each map $i_\alpha:X_\alpha \to X_{\alpha+1}$ in the chain is a pushout
\begin{align*}
\xymatrix{\coprod_\D QCom(A_d) \ar[r] \ar[d] & \coprod_\D QCom(B_d) \ar[d] \\
X_\alpha \ar[r]^{i_\alpha} & X_{\alpha+1}}
\end{align*}

We prove $\eta_X$ is a weak equivalence by transfinite induction. For the base case, $X = S$ is the initial $QCom$-algebra and $\eta_X$ is an isomorphism. Suppose now that $X_\alpha$ is a cellular $QCom$-algebra and $\eta_{X_\alpha}$ is a weak equivalence. Applying $\phi^* \phi_!$ to the diagram above yields
a cube in which the back face is the square above and the front face is the pushout square
\begin{align*}
\xymatrix{\coprod_\D Com(A_d) \ar[r] \ar[d] & \coprod_\D Com(B_d) \ar[d] \\
\phi^*\phi_! X_{\alpha} \ar[r] & \phi^* \phi_! X_{\alpha+1}}
\end{align*}

We will use the cube lemma (in $\M$) to deduce that $\eta_{X_{\alpha+1}}$ is a weak equivalence. First, we have weak equivalences between the elements in the upper left, upper right, and lower left of the two squares, by the case above and the inductive hypothesis. Next, all objects in the cube above are cofibrant in $\M$. For the top face, this is because the domains and codomains of the generating cofibrations for $\M$ are cofibrant, and we have already argued that for such $A_d,B_d$ the objects in the cube above are cofibrant. For the bottom face, this follows from the inductive hypothesis and from hypothesis (4) in Theorem \ref{thm:rect-ax-implies-rect}. For the case where $S$ is cofibrant in $\M$, $X_{\alpha}$ and $X_{\alpha+1}$ are cofibrant because cofibrations with cofibrant source forget to cofibrations in $\M$. For the case where $\M$ is spectra, all horizontal maps are injections and so we can use the cube lemma in the injective model structure. Similarly, the horizontal maps in both squares are cofibrations in $\M$ (resp. injections for the case of spectra), by our hypothesis about cofibrations forgetting. Thus, the cube lemma demonstrates that $\eta_{X_{\alpha+1}}$ is a weak equivalence.

For the limit ordinal case, assume $i_{\gamma}$ is a weak equivalence between cofibrant objects for all $\gamma < \alpha$. We have a map of sequences
\begin{align*}
\xymatrix{
S \ar[r] \ar[d] & X_1 \ar[r] \ar[d] & \dots \ar[r] & X_\gamma \ar[d] \ar[r] & \dots \\
\phi^* \phi_! S \ar[r] & \phi^* \phi_! X_1 \ar[r] & \dots \ar[r] & \phi^* \phi_! X_{\gamma} \ar[r] & \dots}
\end{align*}
All vertical maps are weak equivalences and all horizontal maps are cofibrations between cofibrant objects (resp. injections in the case of spectra). Thus, $i_\alpha$ is a weak equivalence by Proposition 17.9.1 in \cite{hirschhorn}. This completes our proof that $(\phi_!,\phi^*)$ is a Quillen equivalence.
\end{proof}

\begin{corollary} \label{cor:rect-for-spectra}
Let $\M$ be the positive flat stable model structure on symmetric spectra or orthogonal spectra. Let $O$ be an $E_\infty$-operad. Then $O$-alg is Quillen equivalent to $Com$-alg.

Let $\M$ be orthogonal spectra with the positive flat stable model structure. Let $O$ be an $E_\infty$-operad. Then $O$-alg is Quillen equivalent to $Com$-alg.
\end{corollary}

\begin{proof}
For symmetric spectra $\D = sSet$ and for orthogonal spectra $\D=Top$. In each case, the unit of $\D$ is cofibrant. Each of these categories of spectra has domains of the generating cofibrations cofibrant, since generating cofibrations are obtained from $sSet$ and $Top$, where they are maps from spheres into disks. Each of these model categories of spectra satisfies the strong commutative monoid axiom, as will be shown in Section \ref{sec:examples} below. Each satisfies the monoid axiom, as has been shown in \cite{shipley-positive} and \cite{stolz-thesis}, among other places. Each has domains of the generating cofibrations satisfying the requisite smallness hypotheses from (3) in Theorem \ref{thm:rect-ax-implies-rect}. For symmetric spectra this is because all objects are small. For orthogonal spectra this is because the domains are small relative to inclusions and both morphisms of the form $I\otimes \M$ and of the form $QCom(I)$ are closed inclusions (hence transfinite compositions of pushouts of such maps are closed inclusions, see \cite{white-topological}).

Theorem \ref{thm:commMonModel} implies the category of commutative monoids inherits a model structure. Proposition 4.3 in \cite{BM03} implies that $QCom$ is $\Sigma$-cofibrant, since the unit of $\D$ is cofibrant. Theorem 12.1.4 in \cite{fresse-book} implies that $QCom$-alg inherits a transferred semi-model structure from $\M$. We must show that a cofibration of $QCom$-algebras with cofibrant source forgets to a cofibration in $\M$. 

in the former case using , and in the latter case by carrying out the proof of Proposition 4.1 in \cite{shipley-positive} in QCom-alg. The key point is that the free QCom-algebra functor preserves monomorphisms. This is true for the free P-algebra functor for any P, since coproducts and passage to $\Sigma_n$-coinvariants preserve monomorphisms.

Finally, each of these model categories has been shown to satisfy the rectification axiom in existing results in the literature. For the case of symmetric spectra this appears in \cite{shipley-positive}. For orthogonal spectra, this is in \cite{MMSS}. Thus, $QCom$-alg is Quillen equivalent to $Com$-alg. Since the unit of $\D$ is cofibrant, $QCom$ is a $\Sigma$-cofibrant operad weakly equivalent to $Com$, hence weakly equivalent to $O$. Since both $O$ and $QCom$ are $\Sigma$-cofibrant, it follows from \cite{fresse-book} (Theorem 12.5.A) that $O$-algebras are Quillen equivalent to $QCom$-algebras, hence to $Com$-algebras.
\end{proof}

An analogous result is true for equivariant orthogonal spectra, but its proof would take us too far afield. It will be proven in a forthcoming paper with Javier Guti\'{e}rrez. The equivariant analogue of the rectification axiom (mentioned above Lemma \ref{lemma:coprod-of-weak-equivs}) is proven to hold in \cite{kervaire-arxiv}, while rectification with a cofibrant $N_\infty$-operad is proven in the appendix of \cite{blumberg-hill}.

\begin{corollary}
Let $\M$ be the positive stable model structure on symmetric spectra or orthogonal spectra. Let $O$ be an $E_\infty$-operad. Then $O$-alg is Quillen equivalent to $Com$-alg.

Let $\M$ be orthogonal spectra with the positive stable model structure. Let $O$ be an $E_\infty$-operad. Then $O$-alg is Quillen equivalent to $Com$-alg.
\end{corollary}

\begin{proof}
For these examples commutative monoids inherit a model structure, but the strong commutative monoid axiom does not hold. A cofibration of commutative monoids whose domain is cofibrant in $\M$ only forgets to a positive flat cofibration in $\M$ and not a positive cofibration. However, the only place in the proof of Theorem \ref{thm:rect-ax-implies-rect} that we used the strong commutative monoid axiom was the step in which we applied the cube lemma. If we apply the cube lemma in the positive flat stable model structure then we prove that $X\to \phi^* \phi_! X$ is a weak equivalence in the positive flat stable model structure. Thankfully, such maps are precisely the stable equivalences, so $X\to \phi^* \phi_! X$ is a weak equivalence in the positive (non-flat) stable model structure as well.

The rest of the hypotheses of Theorem \ref{thm:rect-ax-implies-rect} are satisfied, as can be seen in \cite{MMSS} and \cite{kervaire-arxiv}. The rectification axiom in the positive (non-flat) stable model structure is implied by the rectification axiom in the positive flat stable model structure, since the weak equivalences agree and every positive cofibrant $X$ is positive flat cofibrant. So the corollary now follows from the proof of Theorem \ref{thm:rect-ax-implies-rect} using the argument of the preceding paragraph to prove that $X\to \phi^* \phi_! X$ is a weak equivalence.
\end{proof}

\begin{corollary}
Let $k$ be a field of characteristic zero. Then the category of commutative differential graded algebras over $k$ is Quillen equivalent to the category of $E_\infty$-algebras.
\end{corollary}

\begin{proof}
The model category $\M=Ch(k)$ satisfies the strong commutative monoid axiom, as will be shown in Section \ref{sec:examples} below. It satisfies the monoid axiom (see \cite{SS00}), has domains of the generating cofibrations cofibrant, and has all objects small (see \cite{hovey-book}), and satisfies the rectification axiom (see \cite{quillen-rational-annals}).
\end{proof}

We pause now to record a proposition about the interplay between the rectification axiom and the commutative monoid axiom which we shall use in Section \ref{sec:examples}.

\begin{prop} \label{prop:rectification-implies-sym}
Suppose $\M$ is a monoidal model category satisfying the rectification axiom.
Then $\Sym^n(-)$ takes trivial cofibrations between cofibrant objects to weak equivalences.
\end{prop}

\begin{proof}
Let $f:A\to B$ be a trivial cofibration between cofibrant objects. Note that $f^{\otimes n}:(A)^{\otimes n} \to (B)^{\otimes n}$ is a trivial cofibration in $\M$ because it is the composite $A^{\otimes n} \to A^{\otimes n-1}\otimes B \to A^{\otimes n-2}\otimes B^{\otimes 2} \to \dots \to B^{\otimes n}$. This follows by iteratively applying the fact that $A\otimes -$ and $B\otimes -$ are left Quillen functors.

Furthermore, $Q_{\Sigma_n} S$ is $\Sigma_n$-cofibrant and so when we take the pushout product of $A^{\otimes n}\to B^{\otimes n}$ with $\emptyset \to Q_{\Sigma_n}S$ we obtain a $\Sigma_n$-trivial cofibration, e.g. by Lemma 2.5.2 in \cite{BM06}. When we pass to $\Sigma_n$-coinvariants we obtain a trivial cofibration $A^{\otimes n} \otimes_{\Sigma_n} Q_{\Sigma_n}S \to B^{\otimes n} \otimes_{\Sigma_n} Q_{\Sigma_n}S$ because $(-)/\Sigma_n$ is left Quillen. Consider the following commutative square, where the bottom horizontal map is $\Sym^n(f)$, the top horizontal map is the map we have just described, and the vertical maps are induced by $Q_{\Sigma_n}S\to S$ and by passage to $\Sigma_n$-coinvariants:
\begin{align*}
\xymatrix{ Q_{\Sigma_n} \otimes_{\Sigma_n} A^{\otimes n} \ar[r] \ar[d] & Q_{\Sigma_n} \otimes_{\Sigma_n} B^{\otimes n} \ar[d] \\
A^{\otimes m}/\Sigma_m \ar[r] & B^{\otimes m}/\Sigma_m }
\end{align*}

We have shown the top vertical map is a weak equivalence. The vertical maps are weak equivalences by the rectification axiom. Thus, the bottom horizontal map is a weak equivalence by the two-out-of-three property.

\end{proof}

In situations arising from topology, where $\M$ is spectra and $\D$ is spaces, the map $Q_{\Sigma_n} S\to S$ is the cofibrant replacement of the point and so is $E\Sigma_n \to *$ in the unpointed setting and $(E\Sigma_n)_+ \to S^0$ in the pointed setting. This proposition is used in Section \ref{sec:examples} to make sure that a particular Bousfield localization respects the commutative monoid axiom. 

We have not undertaken a general study of when rectification between $Com$ and $E_\infty$ holds. The interested reader is encouraged to consult \cite{gutierrez-vogt}, \cite{sagave-schlichtkrull}, and \cite{dmitri} for more information about rectification for general model categories.

\subsection{Relationship to Bousfield Localization}

We now record a few facts regarding the relationship between the model category axioms we have discussed and (left) Bousfield localization. These results are proven in the author's thesis \cite{white-thesis}, and have appeared in the companion paper \cite{white-localization}. Taken together, the following three results give a list of checkable conditions on a model category $\M$ and a set of maps $C$ so that the Bousfield localization $L_C(\M)$ of $\M$ with respect to $C$ satisfies the necessary hypotheses of Theorem \ref{thm:commMonModel}, i.e. so that one may obtain a model structure on the category of commutative monoids in $L_C(\M)$. It is proven in \cite{white-thesis} that these properties imply that commutative $R$-algebras are preserved by $L_C$. Throughout we assume that the maps in $C$ are cofibrations between cofibrant objects. If they are not, then this can be arranged without loss of generality by taking cofibrant replacements of the maps in $C$ and applying the factorization axiom to obtain cofibrations between cofibrant objects.

\begin{theorem} \label{thm:PPaxiom}
Let $\M$ be a left proper, monoidal model category where cofibrant objects are flat and such that the domains of the generating cofibrations are cofibrant. Let $C$ be a set of maps such that the Bousfield localization $L_C(\M)$ exists. Then $L_C(\M)$ has cofibrant objects flat and satisfies the pushout product axiom if and only if for all domains and codomains $K$ of the generating cofibrations, maps in $C\otimes id_K$ are $C$-local equivalences.

Furthermore, without the hypothesis on the domains of the generating cofibrations, we have:

$L_C(\M)$ has cofibrant objects flat and satisfies the pushout product axiom if and only if for all cofibrant $K$, maps in $C \otimes id_K$ are $C$-local equivalences.
\end{theorem}

Note in particular that under these hypotheses $L_C(\M)$ also satisfies the unit axiom. In light of this characterization, we refer to Bousfield localizations satisfying the hypotheses of the theorem as \textit{monoidal Bousfield localizations}. We turn next to the strong commutative monoid axiom, for which we have two preservation results with differing hypotheses. 

\begin{theorem}\label{thm:loc-preserves-cmon-axiom}
Suppose $\M$ is a simplicial model category satisfying the strong commutative monoid axiom. Suppose that for all $n \in \N$ and $f\in C$, $\Sym^n(f)$ is a $C$-local equivalence. Then $L_C(\M)$ satisfies the strong commutative monoid axiom.
\end{theorem}

\begin{theorem}
Assume $\M$ is a monoidal model category satisfying the strong commutative monoid axiom and in which the domains of the generating cofibrations are cofibrant. Suppose that $L_\C(\M)$ is a monoidal Bousfield localization with generating trivial cofibrations $J_\C$. If $\Sym^n(f)$ is a $\C$-local equivalence for all $n \in \N$ and for all $f \in J_\C$, then $L_\C(\M)$ satisfies the strong commutative monoid axiom.
\end{theorem}

Because the results in \cite{white-localization} are general enough to hold only in the presence of a semi-model structure on commutative monoids, it is enough for localization to preserve the pushout product axiom and the commutative monoid axiom. However, we also have a result regarding preservation of the monoid axiom which we record here for the reader's convenience. First we must introduce a new definition, taken from \cite{batanin-berger}: 

\begin{defn} \label{defn:h-monoidal}
A map $f:X\to Y$ is called an \textit{$h$-cofibration} if the functor $f_!:X/\M \to Y/\M$ given by cobase change along $f$ preserves weak equivalences. Formally, this means that in any diagram as below, in which both squares are pushout squares and $w$ is weak equivalence, then $w'$ is also a weak equivalence:
\[
\xymatrix{X \ar[r] \ar[d]_f & A \ar[r]^w \ar[d] & B\ar[d]\\
Y \ar[r] & A' \ar[r]_{w'} & B'}
\]

$\M$ is said to be \textit{$h$-monoidal} if for each (trivial) cofibration $f$ and each object $Z$, $f\otimes Z$ is a (trivial) $h$-cofibration.
\end{defn}

If $\M$ is left proper, then an equivalent characterization of an $h$-cofibration is as a map $f$ such that every pushout along $f$ is a homotopy pushout (this version of the definition above was independently discovered in \cite{white-thesis}). In \cite{batanin-berger}, $h$-monoidality is verified for the model categories of topological spaces, simplicial sets, chain complexes over a field (with the projective model structure), symmetric spectra (with the stable projective model structure), and several other model categories not considered in this paper. More examples can be found in \cite{white-localization}. 

With this definition in hand, it is proven in Proposition 2.5 of \cite{batanin-berger} that if $\M$ is $h$-monoidal and the weak equivalences in $(\M \otimes I)$-cell are closed under transfinite composition, then $\M$ satisfies the monoid axiom. We strengthen this result by replacing the third condition with the hypothesis that the (co)domains of $I$ are finite relative to the class of $h$-cofibrations (here finite means small relative to all limit ordinals, as in Section 7.4 of \cite{hovey-book}). Because this is a statement phrased entirely in terms of $I$, it is preserved by any Bousfield localization $L_C$. We therefore are able to prove:

\begin{theorem}
Suppose $\M$ is an $h$-monoidal model category such that the (co)domains of $I$ are finite relative to the $h$-cofibrations, the domains of the generating cofibrations are cofibrant, and cofibrant objects are flat. Then for any monoidal Bousfield localization $L_C$, the model category $L_C(\M)$ satisfies the monoid axiom.
\end{theorem}

\subsection{Left Properness}
\label{subsec:left-proper}

In \cite{batanin-berger} conditions are provided so that if $\M$ is left proper then the transferred model structure on algebras over a certain type of monad $T$ is left proper. A standard condition in \cite{batanin-berger}, that subsumes the smallness hypothesis in Theorem \ref{thm:commMonModel}, is that $\M$ is compactly generated, i.e. all objects are small relative to $(\M \otimes I)$-cell and the weak equivalences are closed under filtered colimits along morphisms in $(\M \otimes I)$-cell (i.e. the class of weak equivalences is perfect with respect to $(\M \otimes I)$-cell). Unfortunately, the meaning of compactly generated and of $h$-cofibration in this paper and in \cite{batanin-berger} is totally unrelated to the meaning in \cite{EKMM}. The approach in \cite{EKMM}, via the Cofibration Hypothesis, is meant to avoid the need for the monoid axiom, but does not yield left proper model structures on categories of algebras. 

Theorem 3.1 in \cite{batanin-berger} proves that the model structure on monoids constructed in \cite{SS00} is left proper if $\M$ is compactly generated and if the weak equivalences in $\M$ are closed under $\otimes$ (this condition is referred to as $\M$ being \textit{strongly $h$-monoidal}). Following this proof method, we can prove that our model structure on commutative monoids is left proper under the weaker hypothesis that $\M$ is only $h$-monoidal (in the sense of Definition \ref{defn:h-monoidal}). However, because we still need certain monoidal products of weak equivalences to be weak equivalences, we replace the strong $h$-monoidality by the assumptions that cofibrant objects are flat in $\M$ and that the domains of the generating cofibrations are cofibrant, as in the previous section.

\begin{theorem} \label{thm:left-proper}
Let $\M$ be a compactly generated $h$-monoidal model category satisfying the strong commutative monoid axiom and the monoid axiom. Assume the domains of the generating cofibrations in $\M$ are cofibrant and that cofibrant objects are flat. Let $R$ be a commutative monoid in $\M$. Then the category $CAlg(R)$ inherits a left proper transferred model structure from $\M$. In particular, when $R=S$ this gives a left proper model structure on commutative monoids in $\M$.
\end{theorem}

\begin{proof}
As usual we can reduce to proving the case $R=S$, since $CAlg(R)=R\downarrow CMon(\M)$ and an under-category is left proper if when we forget to $CMon(\M)$ the result is left proper. 
The hypotheses of the theorem subsume those of Theorem \ref{thm:commMonModel}, so we may assume $\CMon(\M)$ admits a transferred model structure. Note also that $\M$ is left proper, since $h$-monoidality implies left properness. 
Following Theorem 2.14 in \cite{batanin-berger}, what must be shown is that for any cofibration $u:K\to L$ in $\M$ and for any weak equivalence $f:A\to B$ in $\CMon(\M)$ with a map $\alpha:K\to U(A)$, the map $A[u,\alpha] \to B[u,f\alpha]$ defined by the following diagram is a weak equivalence:

\begin{align*}
\xymatrix{\Sym(K) \ar[d]_{\Sym(u)} \ar[r]^{\Sym(\alpha)} \po & \Sym(U(A)) \ar[d] \ar[r] \po & A\ar[d] \ar[r]^f \po & B\ar[d] \\
\Sym(L) \ar[r] & \Sym(P) \ar[r] & A[u,\alpha] \ar[r] & B[u,f\alpha]}
\end{align*}

Here $P$ is the pushout of $L\gets K\to U(A)$ in $\M$, the left-hand square is obtained by applying $\Sym$ to this pushout, and the map $\Sym(U(A))\to A$ is the structure map of the monad $\Sym$. The notation $A[u,\alpha]$ and $B[u,f\alpha]$ are defined by the pushout diagrams above.

In order to prove that $A[u,\alpha] \to B[u,f\alpha]$ is a weak equivalence we observe as in Theorem 2.14 of \cite{batanin-berger} that the filtration on each component induces a filtration of the map

\begin{align*}
\xymatrix{A[u]^{(0)} \ar[r] \ar[d] & A[u]^{(1)} \ar[r] \ar[d] & \dots \ar[r] & \colim_n A[u]^{(n)} \ar@{=}[r] \ar[d] & U(A[u,\alpha]) \\
B[u]^{(0)} \ar[r] & B[u]^{(1)} \ar[r] & \dots \ar[r] & \colim_n B[u]^{(n)} \ar@{=}[r] & U(B[u,f\alpha])}
\end{align*}

We have changed notation, but the filtration across the top line is precisely the filtration denoted by $P_0\to P_1\to \dots$ in Theorem \ref{thm:commMonModel} for the diagram $A\gets \Sym(K)\to \Sym(L)$, and the filtration across the bottom is the corresponding filtration for $B$. In particular, the horizontal maps across the top are pushouts of maps of the form $A\otimes u^{\boxprod n}/\Sigma_n$ and the horizontal maps across the bottom are pushouts of maps of the form $B\otimes u^{\boxprod n}/\Sigma_n$. The strong commutative monoid axiom guarantees us that $u^{\boxprod n}/\Sigma_n$ is a cofibration, and $h$-monoidality tells us that the horizontal maps are in the class $(\M \otimes I)$-cell. Since $\M$ is compactly generated, any filtered colimit of weak equivalences along such maps is a weak equivalence, so we are reduced to proving the vertical maps are weak equivalences. This will be accomplished by induction, using the fact that the vertical maps may be realized inductively as colimits of the following cubes:

\begin{align*}
\xymatrix{A\otimes Q_n/\Sigma_n \ar[rr] \ar[dd] \ar[dr] && A \otimes L^{\otimes n}/\Sigma_n \ar[dd] \ar[dr] & \\
& A[u]^{(n-1)} \ar[rr] \ar[dd] && A[u]^{(n)} \ar[dd] \\
B\otimes Q_n/\Sigma_n \ar[rr] \ar[dr] && B \otimes L^{\otimes n}/\Sigma_n \ar[dr] & \\
& B[u]^{(n-1)} \ar[rr] && B[u]^{(n)}}
\end{align*}

Since $f$ is a weak equivalence in $\CMon(\M)$, it is a weak equivalence in $\M$, so $A[u]^{(0)} \to B[u]^{(0)}$ is simply the weak equivalence $f$. By induction we may assume $A[u]^{(n-1)} \to B[u]^{(n-1)}$ is a weak equivalence. We must now prove that the other vertical maps are weak equivalences. This is where we use our hypotheses on $\M$. Our assumption on the domains of the generating cofibrations implies $K$ and $L$ are cofibrant, and hence that the maps inside the cube defining $Q_n$ are cofibrations (hence $h$-cofibrations because $\M$ is left proper). Since passage to $\Sigma_n$-coinvariants commutes with pushout, and because $\M$ satisfies the commutative monoid axiom, $Q_n/\Sigma_n$ and $L^{\otimes n}/\Sigma_n$ are cofibrant. A detailed proof of this claim is given in Lemma \ref{lemma:boxprod-equiv-if-sym}. Because cofibrant objects are flat, the vertical maps $f\otimes Q_n/\Sigma_n$ and $f\otimes L^{\otimes n}/\Sigma_n$ are weak equivalences.

The strong commutative monoid axiom and $h$-monoidality of $\M$ imply that $A\otimes u^{\boxprod n}/\Sigma_n$ and $B\otimes u^{\boxprod n}/\Sigma_n$ are $h$-cofibrations. The characterization of $h$-cofibrations in a left proper model category (given after Definition \ref{defn:h-monoidal} above) implies that both the top and bottom squares of the cube above are homotopy pushout squares. As in the proof of Theorem 3.1 in \cite{batanin-berger} this implies that the back square in the cube above is a homotopy pushout square, and so the front one is too. Finally, this implies that $A[u]^{(n)}\to B[u]^{(n)}$ is a weak equivalence, completing our induction and the proof that $\CMon(\M)$ is left proper.

\end{proof}

As in the proof of Theorem 3.1 in \cite{batanin-berger}, the proof above also has a relative version when $\M$ fails to be $h$-monoidal.

\begin{theorem}
Let $\M$ be a compactly generated monoidal model category satisfying the strong commutative monoid axiom and the monoid axiom. Let $R$ be a commutative monoid in $\M$. Then the category $CAlg(R)$ inherits a relatively left proper transferred model structure from $\M$.
\end{theorem}

Here relatively left proper means that the pushout by a cofibration in $CAlg(R)$ of any weak equivalence $f:A\to B$ where $U(A)$ and $U(B)$ are cofibrant in $\M$ is again a weak equivalence.

\begin{proof}
The hypotheses of the theorem still imply that $\CMon(\M)$ inherits a model structure. The proof proceeds precisely as above, but now one may assume $U(A)$ and $U(B)$ are cofibrant in $\M$, and that $u$ has a cofibrant domain (this is why we do not need a tractability hypothesis on $\M$). So maps of the form $A\otimes u^{\boxprod n}/\Sigma_n$ and $B\otimes u^{\boxprod n}/\Sigma_n$ are cofibrations. Furthermore, all objects of the cube above are cofibrant and we no longer need the hypothesis that cofibrant objects are flat in order to conclude that the vertical maps are weak equivalences. We simply use Ken Brown's lemma, since $-\otimes Z$ will preserve weak equivalences between cofibrant objects for any cofibrant $Z$ (e.g. $Z=Q_n/\Sigma_n$ or $Z=L^{\otimes n}/\Sigma_n$). Finally, the Cube Lemma (Lemma 5.2.6 in \cite{hovey-book}) completes the induction and implies $A[u]^{(n)}\to B[u]^{(n)}$ is a weak equivalence.
\end{proof}

\subsection{Lifting Quillen Equivalences} \label{subsec:liftingQEs}

We turn now to the question of when a monoidal Quillen equivalence $F:\M \to \cat{N}$, between model categories satisfying the commutative monoid axiom, induces a Quillen equivalence of categories of commutative monoids. 

\begin{theorem}
Suppose $\M$ and $\mathcal{N}$ satisfy the commutative monoid axiom. Suppose $F:\M \to \mathcal{N}$ is left Quillen equivalence and a strong symmetric monoidal functor. Let $T$ be a commutative monoid which is cofibrant in $\M$ and such that cofibrant commutative $T$-algebras forget to cofibrant objects in $\M$. Then $F(T)$ is a commutative monoid in $\mathcal{N}$ and the functor $\tilde{F}:$CAlg$(T)\to$ CAlg($F(T)$), induced by $F$, is a left Quillen equivalence.
\end{theorem}

Note that the hypothesis about cofibrant commutative $T$-algebras forgetting to cofibrant objects in $\M$ can be arranged either by assuming the strong commutative monoid axiom and that the unit $S$ of $\M$ is cofibrant (in which case the hypothesis holds by Corollary \ref{cor:cofibrants-forget}), or by working in the setting where $\M$ is a positive flat model structure on a monoidal category of spectra (in which case the hypothesis was proved to hold in \cite{shipley-positive}).

\begin{proof}
First, we lift the functor $F$ to a functor of commutative $T$-algebras. Let $S$ denote the unit of $\M$ and let $S_{\mathcal{N}}$ denote the unit of $\mathcal{N}$. For any commutative $T$-algebra $M$, $FM$ is a commutative $FT$-algebra with structure maps $FT \otimes FM \cong F(T\otimes M) \to FM$, $FM \otimes FM \cong F(M\otimes M)\to FM$, and $S_{\mathcal{N}}\otimes FM \cong FS \otimes FM \cong F(S\otimes M)\to FM$ inherited from $M$. For commutativity of $FM$ we use that $F$ is \textit{symmetric} monoidal.

Let $\eta:X\to UFX$ and $\epsilon:FUX\to X$ be the unit and counit of the adjunction $(F,U)$. We now show that $U$ lifts to a functor of commutative monoids. For any commutative $FT$-algebra $N$, $UN$ is a $T$-module with structure map $T\otimes UN \to UFT \otimes UN \to U(FT\otimes N)\to UN$ where the first map is $\eta \otimes 1$ and the second map is adjoint to $F(UFT \otimes UN)\to FUFT\otimes FUN \to FT\otimes N$, the composite of a natural isomorphism with $\epsilon \otimes \epsilon$. Similarly, $UN$ is a commutative $T$-algebra with structure maps $UN \otimes UN \to U(N\otimes N) \to U(N)$ where the first map is adjoint to $F(UN\otimes UN)\cong FUN \otimes FUN \to N\otimes N$ again using $\epsilon \otimes \epsilon$. Lastly, the unit $S_{\mathcal{N}}\cong FS\to N$ is adjoint to $S\to UN$. 
It is easy to verify that $F$ and $U$ remain adjoint as functors of commutative monoids.

Since $U$ preserves fibrations and trivial fibrations of commutative monoids (since it does so as a functor $\cat{N}\to \M$), it is a right Quillen functor. To prove that $(F,U)$ forms a Quillen equivalence, we must prove that for any cofibrant commutative $T$-algebra $X$, the natural map $X\to UR_TFX$ is a weak equivalence, where $R_T$ is a fibrant replacement functor in the category of commutative $FT$-algebras. Let $R$ be a fibrant replacement functor in $\cat{N}$. Then there is a weak equivalence $RFX \to R_TFX$ because $R_TFX$ is fibrant in $\cat{N}$ and weakly equivalent to $FX$. This is a weak equivalence between fibrant objects in $\cat{N}$, so $URFX \to UR_TFX$ is a weak equivalence in $\M$. Since $X$ is a cofibrant commutative $T$-algebra, our hypothesis guarantees $X$ is cofibrant in $\M$. It follows that $X\to URFX$ is a weak equivalence (since $(F,U)$ is a Quillen equivalence), hence that $X\to UR_TFX$ is a weak equivalence in $\M$.
\end{proof}

Note that because we only needed $X\to URFX$ to be a weak equivalence to finish our proof above, it also works in positive (non flat) model structures on monoidal categories of spectra, since it is sufficient that $X$ be positive flat cofibrant rather than positive cofibrant.

\begin{remark}
The proof above is based on Theorems 2.7 and 3.6 in \cite{hovey-monoidal}. However, the Dold-Kan equivalence is not strongly symmetric monoidal (see page 2 of \cite{schwede-shipley-equivalences}). One could attempt to generalize the Theorem above in the way that \cite{schwede-shipley-equivalences} generalized \cite{hovey-monoidal}, and work with weak monoidal Quillen pairs, following Theorem 3.12 of \cite{schwede-shipley-equivalences}, but this would not solve the fact that the Dold-Kan equivalence is not symmetric.
\end{remark}

\section{Examples}
\label{sec:examples}

In this section we verify the strong commutative monoid axiom for the model categories of chain complexes over a field of characteristic zero, for simplicial sets, for topological spaces, and for positive flat model structures on various categories of spectra. We also discuss precisely what is true for positive (non-flat) model structures of spectra. Throughout this section we make use the following lemma, which is proven in Appendix \ref{appendix-generators}.

\begin{lemma} \label{lemma:generators-suffice}
Suppose $\M$ is a cofibrantly generated monoidal model category and that for all $f \in I$ (resp. $J$) we know that $f^{\boxprod n}/\Sigma_n$ is a (trivial) cofibration. Then the strong commutative monoid axiom holds for $\M$.
\end{lemma}

\subsection{Commutative Differential Graded Algebras in characteristic zero} \label{subsec:cdgas}

Consider a field $k$ and $\M = Vect(k)$. Then $\M$ satisfies the strong commutative monoid axiom if and only if char$(k)=0$. 
Because $\M^{\Sigma_n} \cong k[\Sigma_n]-mod$, the projective model structure is nicely behaved (i.e. matches the injective model structure) exactly when $k[\Sigma_n]$ is semisimple, i.e. exactly when $k$ has characteristic zero. Indeed, such $\M$ satisfies the stronger condition required in Theorem \ref{thm:Lurie}. This example generalizes to pertain to $Ch(R)$ whenever $R$ is a commutative $\Q$-algebra.

The commutative monoid axiom fails over $\F_2$ because $\F_2$ is not projective over $\F_2[\Sigma_2]$ (because now Maschke's Theorem does not hold) and so the cokernel of $f^{\boxprod n}$ does not have a free $\Sigma_n$ action, and this will be an obstruction to $f^{\boxprod n}/\Sigma_n$ being a cofibration.

That $CDGA(k)$ cannot inherit a model structure for char$(k)=p>0$ has been known for many years. The fundamental problem is that $\Sym(-)$ does not preserve weak equivalences between cofibrant objects and so cannot be a left Quillen functor. This is because for example $\Sym^p(D(k))$ will not be acyclic even though the disk $D(k)$ is acyclic. 

\subsection{Spaces} \label{subsec:spaces}

\begin{theorem}
The category of simplicial sets satisfies the strong commutative monoid axiom but does not satisfy Lurie's axiom or the rectification axiom.
\end{theorem}

\begin{proof}

To see that the rectification axiom fails, consider $X=\Delta[0]$. Then the rectification axiom is asking $B\Sigma_n$ to be contractible. To see that Lurie's axiom fails, consider $f^{\boxprod 2}$ where $f:S^0\to D^1$. This map is not a $\Sigma_2$-cofibration because the action on the cofiber of $f^{\boxprod 2}$ is not free. However, to show that we get a cofibration after passing to $\Sigma_2$ coinvariants is easy, because the map is a monomorphism. Furthermore, this line of reasoning generalizes to show that $f^{\boxprod n}/\Sigma_n$ is a cofibration whenever $f$ is a generating (trivial) cofibration. To check that it's also a weak equivalence if $f$ is a generating trivial cofibration, we use the following theorem of Casacuberta \cite{eternal-preprint}:

\begin{theorem}
If $f$ is any map in sSet, then $Sym(-)$ preserves $f$-equivalences.
\end{theorem}

Obviously, this proves much more than we needed, and in fact we use the proof of this theorem in \cite{white-localization} to see that any monoidal Bousfield localization of sSet also satisfies the strong commutative monoid axiom. 
The key point in the proof of this theorem is due to an observation of Farjoun \cite{farjoun} which says that for any $X$, $Sym^n(X)$ can be written as a homotopy colimit of a free diagram formed by the orbits of $\Sigma_n$ where each quotient $\Sigma_n/H$ is sent to the fixed-point subspace $(X^n)^H$. It is then not too much work to see that $Sym^n(-)$ preserves weak equivalences (and more generally $f$-equivalences), as is proven in \cite{eternal-preprint}. We then use Lemma \ref{lemma:boxprod-equiv-if-sym} to see that $f^{\boxprod n}/\Sigma_n$ is a weak equivalence whenever $f$ is a trivial cofibration, completing our proof that the strong commutative monoid axiom holds.

\end{proof}

Observe that the counterexample displaying the failure of Lurie's axiom and the rectification axiom also applies to $Top$, $sSet^G$, and $Top^G$.

\begin{theorem}
The category of compactly generated topological spaces satisfies the strong commutative monoid axiom.
\end{theorem}

\begin{proof}

In $Top$, cofibrations are no longer monomorphisms, but the strong commutative monoid axiom still holds. This may be verified by either checking it directly on the generating maps $S^{n-1}\to D^n$ and $D^n\to D^n \times [0,1]$ (a valuable exercise), or by transporting the strong commutative monoid axiom on $sSet$ to $Top$ via the geometric realization functor. From \cite{hovey-monoidal} we see that $Top$ satisfies the necessary smallness hypotheses, so Theorem \ref{thm:commMonModel} applies.
\end{proof}

In case the reader is interested in checking the commutative monoid axiom on $Top$ directly, we remark that the interpretation of Farjoun's work in \cite{eternal-preprint} makes clear that the only property of simplicial sets being used in the argument is that the fixed point subspaces of actions of subgroups of $\Sigma_k$ on $X^k$ are homeomorphic to spaces $X^n$ for some $n\leq k$. So one could apply Farjoun's work just as well in $Top$ as in $sSet$. Indeed, Farjoun's work provides a way to ``free up" any diagram category and view the colimit of a diagram as the homotopy colimit of a different diagram (indexed by the so-called orbit category). In this way good homotopical properties can be achieved in a great deal of generality. The fact that the same argument works in both $Top$ and $sSet$ leads us to make the following conjecture.

\begin{conjecture}
Suppose that $\M$ is a concretizable Cartesian closed model category in which cofibrations are closed under the operation $(-)^{\boxprod n}/\Sigma_n$. Then the strong commutative monoid axiom holds in $\M$.
\end{conjecture}

We now turn to equivariant spaces.

\begin{theorem}
Let $G$ be a finite group. Then $sSet^G$ and $Top^G$ satisfy the strong commutative monoid axiom.
\end{theorem}

\begin{proof}

We begin with $sSet^G$. Note that just as for $sSet$, cofibrations are monomorphisms.
Thus, the same proof as for $sSet$ applies. In particular, when applying Farjoun's trick on $(X^n)^H$ where $H<\Sigma_n$, we simply use the fact that the $G$ action and the $\Sigma_n$ action commute.

To handle the situation of $Top^G$ we may again transfer the strong commutative monoid axiom via geometric realization. Here we really need $G$ to be a finite group. For any simplicial group $G$, a $G$ action on $X\in sSet$ is taken to an action of $|G|$ on $|X|$ by geometric realization. If $G$ is finite then $G=Sing|G|$ acts on $Sing|X|$ and we can prove $sSet^G$ is Quillen equivalent to $Top^{|G|}$. However, for non-finite $G$ we do not know whether or not every subgroup $K$ of the topological group $|G|$ is realized as some $|H|$ for $H < G$, so there may be fewer weak equivalences in $Top^{|G|}$ than in $sSet^G$.

\end{proof}

\subsection{Symmetric Spectra} \label{subsec:symspec}
The obstruction noticed by Gaunce Lewis and discussed in \cite{lewis} guarantees that commutative monoids in the usual model structure on symmetric spectra cannot inherit a model structure, because the unit is cofibrant and because the fibrant replacement functor is symmetric monoidal. This second property cannot be changed, but there are model structures on symmetric spectra in which the unit is not cofibrant. The positive model structure was introduced in \cite{hovey-shipley-smith} and \cite{MMSS} and this model structure breaks the cofibrancy of the sphere by insisting that cofibrations be isomorphisms in level 0 (though in other levels they are the same as the usual cofibrations of symmetric spectra). In \cite{shipley-positive}, Shipley found a more convenient model structure which is now called the positive flat model structure. In this model structure the cofibrations are enlarged to contain the monomorphisms, and then the condition in level 0 is applied. The result is a model structure in which commutative ring spectra inherit a model structure and in which cofibrations of commutative ring spectra forget to cofibrations of spectra.

Note that in \cite{DAG3}, Lurie's axiom is claimed to hold for positive flat symmetric spectra. This is an error, as acknowledged in \cite{MO-error-Lurie}. Indeed, the example given in Proposition 4.2 of \cite{shipley-positive} demonstrates this failure conclusively, for both the positive and the positive flat model structures. We will now show that the commutative monoid axiom holds for positive flat (stable) symmetric spectra, and a slight weakening holds for positive (stable) symmetric spectra.

\subsubsection{Positive Flat Stable Model Structure} \label{subsubsec:posFlatSymSpec}

\begin{theorem} \label{thm:comm-mon-axiom-symmetric-spectra}
The strong commutative monoid axiom holds for the positive flat stable model structure on symmetric spectra.
\end{theorem}

\begin{proof}
By Lemma \ref{lemma:generators-suffice}, it's sufficient to check the strong commutative monoid axiom on the generating (trivial) cofibrations. Such maps are cofibrations between cofibrant objects, so Lemma 8.3.2(1) of \cite{white-yau} implies that for any generating trivial cofibration $g$ and any generating cofibration $f$, the maps $g^{\boxprod n}/\Sigma_n$ and $f^{\boxprod n}/\Sigma_n$ are monomorphisms for all $n$. Similarly, Lemma 8.3.2(2) of \cite{white-yau} implies that $\Sym^n(g)$ is a weak equivalence for all $n$. These two results are special cases of Proposition 4.28$^*$(b) and Proposition 4.29$^*$(a) in \cite{harper-spectra} (corrigendum). In the notation of \cite{white-yau}, $B$ should be taken to be the monoidal unit $S$ with the trivial $\Sigma_n$-action.

Together with the observations above, Lemma A.3 (applied in the injective model structure on symmetric spectra), implies that $g^{\boxprod n}/\Sigma_n$ is a weak equivalence. We are therefore reduced to proving that $g^{\boxprod n}/\Sigma_n$ and $f^{\boxprod n}/\Sigma_n$ are cofibrations rather than only monomorphisms. In light of Lemma \ref{lemma:generators-suffice}, it suffices to do so for generating cofibrations $f$, as the result will then hold for all cofibrations. Thankfully, this has been painstakingly checked by Pereira in Theorems 1.6 and 1.7 of \cite{pereira-cofibrancy}. Pereira introduces the $S$ $\Sigma$-Inj $G$-Proj model structure on the category (Sp$^\Sigma)^G$ of $G$ objects in symmetric spectra, for any finite group $G$. We need the case where $G = \Sigma_n$, and we note that $S$-model structures are the same as flat model structures. Theorem 1.6 of \cite{pereira-cofibrancy} implies that $f^{\boxprod n}$ is a cofibration in this new model structure. It can be deduced from the proof of Theorem 1.7 in Section 4 of \cite{pereira-cofibrancy} that the functor $(-)/\Sigma_n$ from the $S$ $\Sigma$-Inj $\Sigma_n$-Proj model structure to the positive flat model structure preserves cofibrations (this is the point of the projectivity: to free up the $\Sigma_n$-action). It follows that $(-)^{\boxprod n}/\Sigma_n$ preserves positive flat cofibrations, and hence that the strong commutative monoid axiom is satisfied.
\end{proof}

We remark that this theorem together with Lewis's example demonstrate that the commutative monoid axiom need not be preserved by monoidal Quillen equivalences, since the positive flat stable model structure is monoidally Quillen equivalent to the canonical stable model structure. This can be seen via Proposition 2.8 in \cite{shipley-positive}, together with the fact that stable cofibrations are contained in flat cofibrations (Lemma 2.3 in \cite{shipley-positive}) and the fact that the two model structures have the same weak equivalences. We do not know of a similar example which would demonstrate that the monoid axiom need not be preserved by monoidal Quillen equivalence.

\subsubsection{Positive Stable Model Structure} \label{subsubsec:posSymSpec}

Shipley proves in \cite{shipley-positive} that positive symmetric spectra do not satisfy the property that cofibrations of commutative monoids forget to cofibrations of symmetric spectra. Thus, this model structure cannot satisfy the strong commutative monoid axiom. However, Proposition 4.2 in \cite{shipley-positive} proves that a cofibration of commutative $R$-algebras forgets to a positive $R$-cofibration (and hence to an $R$-cofibration) even though it is not a positive cofibration in the sense of \cite{MMSS}. This suggests the following result:

\begin{prop} \label{prop:stable-to-flat-stable}
Let $f$ be a (trivial) cofibration in the positive stable model structure. Then $f^{\boxprod n}/\Sigma_n$ is a (trivial) cofibration in the positive flat stable model structure. Furthermore, commutative monoids inherit a model structure in the positive stable model structure.
\end{prop}

\begin{proof}
The proof is identical to the proof that the positive flat stable model structure satisfies the strong commutative monoid axiom. This is because positive cofibrations form a subclass of positive flat cofibrations. For the statement regarding trivial cofibrations, the same logic used above holds, because it is a Bousfield localization with respect to the same class of maps, and the weak equivalences of both the positive stable and positive flat stable model structures are the same. In particular, this observation proves that the positive (stable) model structures satisfy the weak form of the commutative monoid axiom discussed in Remark \ref{remark:weak-cmon-axiom}, so commutative monoids inherit a model structure.
\end{proof}

Shipley provides a counterexample which demonstrates that $\Sym(F_1S^1)$ is not positively cofibrant (only positively flat cofibrant) because $[(F_1S^1)^{(2)}/\Sigma_2]_2 = (S^1\wedge S^1)/\Sigma_2$ and this is not $\Sigma_2$-free. Thus, Proposition \ref{prop:cofibrations-forget} cannot hold as stated. However, for the same reasons as in the proof above (namely, the containment of positive cofibrations in positive flat cofibrations) we can obtain the following weakened form of Proposition \ref{prop:cofibrations-forget}.

\begin{proposition} \label{prop:cofibrations-forget-weakened}
Let $\M$ be the positive stable model structure on symmetric spectra, and let $CAlg(R)$ be the model structure passed from $\M$ to the category of commutative $R$-algebras (where $R$ is a commutative monoid in $\M$). Suppose $f$ is a cofibration in $CAlg(R)$ whose source is cofibrant in $\M$. Then $f$ forgets to a cofibration in the positive flat stable model structure. 
\end{proposition}

\subsection{General Diagram Spectra} \label{subsec:diagram-spectra}

In \cite{MMSS}, a general theory of diagram spectra is introduced which unifies the theories of $\mathbb{S}$-modules, symmetric spectra, orthogonal spectra, $\Gamma$-spaces, and $W$-spaces. For the first, homotopy-coherence is built into the smash product, so commutative monoids immediately inherit a model structure and there is rectification between $Com$-alg and $E_\infty$-alg. For the next two, positive model structures are introduced which allow strictly commutative monoids to inherit model structures. The rectification axiom is then proved and rectification is deduced as a result.

\begin{theorem}
The positive flat stable model structure on (equivariant) orthogonal spectra satisfies the strong commutative monoid axiom and (the equivariant analogue of) the rectification axiom. The positive stable model structure satisfies the weak commutative monoid axiom, Proposition \ref{prop:stable-to-flat-stable}, and Proposition \ref{prop:cofibrations-forget-weakened}. 
\end{theorem}

\begin{proof}

For the positive flat stable model structure on (equivariant) orthogonal spectra, proceed as in the proof of Theorem \ref{thm:comm-mon-axiom-symmetric-spectra}, but using (equivariant) topological spaces rather than simplicial sets. The rectification axiom is proven in \cite{MMSS} (and in \cite{blumberg-hill} for the equivariant case). For the positive stable model structure proceed as in Proposition \ref{prop:stable-to-flat-stable} and Proposition \ref{prop:cofibrations-forget-weakened}.
\end{proof}

We turn now to $W$-spaces and $\Gamma$-spaces. Recall that $W$ is the category of based spaces homeomorphic to finite CW-complexes, $\Gamma$ is the category of finite based sets, and $\mathscr{D}$-spaces are functors from $\mathscr{D}$ to $Top$ (where $\mathscr{D}$ is either $W$ or $\Gamma$). The indexing category for $\Gamma$-spaces is a subset of $W$. First, Lewis's counterexample \cite{lewis} still applies to rule non-positive model structures out from consideration. This is discussed in the context of $\Gamma$-spaces in Remark 2.6 of \cite{schwede-gamma}. The author has not been able to find a place where this is written down for $W$-spaces, but it is clear that the same counterexample applies for $W$-spaces. We must work in positive model structures on $W$-spaces and $\Gamma$-spaces. Such positive model structures are introduced in Section 14 of \cite{MMSS}, where their monoidal properties are also discussed. 

Analogously to the setting of symmetric and orthogonal spectra, one can define positive \textit{flat} model structures (also known as convenient model structures) on $\Gamma$-spaces and $W$-spaces. For instance, one can carry out the program of \cite{shipley-positive} for $\Gamma$-spaces (e.g. following the work in \cite{sagave-spectra-of-units-positive-gamma} and making use of the relationship between $\Gamma$-spaces and symmetric spectra as explored in \cite{schwede-book-symmetric-spectra}) to obtain the necessary mixed model structure on spaces. From there it is purely formal to construct the appropriate levelwise model structure on diagrams, e.g. using Theorem 6.5 in \cite{MMSS}. The weak equivalences are the level equivalences and the generating cofibrations for $W$-spaces take the form $F_WI = \{F_d(i)\;|\;d\in \skel{W}, i\in I\}$ where $F_d(-)$ is $W(d,-):W\to Top$. The indexing category for $\Gamma$-spaces is a subset of $W$, so an analogous construction works for the generating cofibrations of $\Gamma$-spaces.

Passage from the levelwise structure to the positive flat model structure is again formal, and is accomplished by truncating the levelwise cofibrations to force levelwise cofibrations to be isomorphisms in degree 0. Finally, passage to the positive flat stable model structure may be accomplished via Bousfield localization, just as in Section 8 of \cite{MMSS}.

\begin{theorem}
The positive flat model structures on $W$-spaces and $\Gamma$-spaces satisfy the strong commutative monoid axiom. The positive model structure on $W$-spaces and $\Gamma$-spaces satisfies the weak commutative monoid axiom. So commutative monoids inherit model structures in both settings.
\end{theorem}

The verification of the strong commutative monoid axiom proceeds precisely as for the positive flat model structure on symmetric spectra. In particular, one can reduce the verification to a verification in spaces. We leave the details to the reader. The difficulty comes in the part of the proof when one attempts to pass the commutative monoid axiom to the stable model structure, and that is why the adjective stable is not in the statement of the theorem. In particular, the difficulty is that the rectification axiom is not known to hold for $\mathscr{D}$-spaces (where $\mathscr{D}$ is either $W$ or $\Gamma$). Indeed, we can show that the rectification axiom cannot hold.

First, if the rectification axiom held, then the proof that the strong commutative monoid axiom holds for positive flat stable symmetric spectra (i.e. via Theorem \ref{thm:loc-preserves-cmon-axiom}) would prove that $\mathscr{D}$-spaces satisfy the commutative monoid axiom. Secondly, because of the rectification axiom the rest of the work in \cite{MMSS} and \cite{shipley-positive} would prove that commutative $\mathscr{D}$-rings were Quillen equivalent to $E_\infty$-algebras and this would contradict the main theorem of Tyler Lawson's paper \cite{lawson-gamma}. 

Lawson produces an $E_\infty$-algebra in $\Gamma$-spaces which cannot be strictified to a commutative $\Gamma$-ring. Together with the monoidal functor from $\Gamma$-spaces to $W$-spaces (developed in \cite{MMSS}), this same counterexample proves that not all $E_\infty$-algebras in $W$-spaces can be strictifed to commutative $W$-rings.

\subsection{Diagram Categories}

We now investigate conditions on a model category $\M$ and on a small indexing category $\D$ so that the commutative monoid axiom holds for $\M^\D$.

\subsubsection{Injective Model Structures}

The case of injective model structures on diagrams is particularly easy, since weak equivalences and cofibrations are defined in $\M$. By \ref{subsec:spaces}, the category of simplicial sets satisfies the commutative monoid axiom. This property will be inherited by the category of simplicial presheaves with the injective model structure. In particular, we have

\begin{proposition}
Suppose $\M$ is a monoidal model category which satisfies the commutative monoid axiom. Suppose $\D$ is a small category such that the injective model structure on $\M^\D$ exists. Then the injective model structure with the levelwise monoidal structure satisfies the commutative monoid axiom.
\end{proposition}

\begin{proof}
Suppose $f:X\to Y$ is a cofibration in $\M^\D$ so that each $f_d$ is a cofibration in $\M$. Colimits in $\M^\D$ are computed componentwise, so the domain $Q_n$ of $f^{\boxprod n}$ has the property that $(Q_n)_d$ is the domain of $(f_d)^{\boxprod n}$. Furthermore, the same is true of $Q_n /\Sigma_n$ and of $Y^{\otimes n}/\Sigma_n$ for the same reason. Since $\M$ satisfies the commutative monoid axiom, each $(f_d)^{\boxprod n}/\Sigma_n$ is a cofibration in $\M$. Thus, $f^{\boxprod n}/\Sigma_n$ is a levelwise cofibration, i.e. a cofibration in $\M^\D$. The case of a trivial cofibration $f$ is similar, since injective weak equivalences are also defined levelwise.
\end{proof}

In a related vein, we make the following conjecture.

\begin{conjecture}
Any excellent model category in the sense of \cite{htt}, Definition A.3.2.16, will satisfy the commutative monoid axiom.
\end{conjecture}

The results in \ref{subsec:spaces}, together with the main results of \cite{white-localization} imply that any left Bousfield localization of simplicial sets will satisfy the commutative monoid axiom. A similar statement is true for the injective model structure on simplicial presheaves, but one must be more careful that the pushout product axiom is preserved by the localization (i.e. not all Bousfield localizations are monoidal in the sense of \cite{white-localization}). In recent joint work with Michael Batanin, we have shown that all monoidal Bousfield localizations of simplicial presheaves satisfy the hypotheses of Theorem \ref{thm:loc-preserves-cmon-axiom} and as a result the model structures considered by Rezk in \cite{rezk-theta-n} will satisfy the commutative monoid axiom.

\subsubsection{Projective Model Structure}

The projective model structure is more subtle than the injective model structure, because one has less control over the projective cofibrations. For this reason, we begin with a simple case, but one which we expect to have important applications in the future.

Recall from \cite{hovey-smith-ideals} that for any monoidal model category $\M$, the diagram category $\Arr(\M):=\M^{\bullet \to \bullet}$ may be endowed with a monoidal model structure in which weak equivalences and fibrations are defined componentwise and in which the monoidal product is the pushout product $f\boxprod g$. Hovey proved that if $\M$ is cofibrantly generated and satisfies the monoid axiom then the same is true of $\Arr(\M)$.

When $\M$ is the category of symmetric spectra, monoids are Smith ideals, named after Jeff Smith to whom this definition is due. The main results in \cite{hovey-smith-ideals} together with the model structure on monoids from \cite{SS00} allow for a model category of Smith ideals. Smith envisioned using ideals to extend many results in commutative algebra to the world of ring spectra. A necessary step in carrying out this program is having a good homotopy theory for commutative ideals. This is accomplished by the following result.

\begin{theorem} Suppose $\M$ is a cofibrantly generated monoidal model category. Then the projective model structure on $\Arr(\M)$ with monoidal structure given by the pushout product satisfies the commutative monoid axiom. If $\M$ satisfies the (strong) commutative monoid axiom then so does $\Arr(\M)$.
\end{theorem}

\begin{proof}
By Appendix \ref{appendix-generators} it is sufficient to check the commutative monoid axiom and the strong commutative monoid axiom on generating (trivial) cofibrations. Recall that the model structure on $\Arr(\M)$ is transferred from $\M^{\bullet \; \; \bullet} = \M\times \M$. Let $\D$ be the walking arrow category $\bullet \to \bullet$. The generating cofibrations for $\Arr(M)$ are $F(I')$ and $F(J')$ where $I'$ and $J'$ are generators for $\M^{\D^{disc}} = \prod_{d\in \D} \M$ and $F$ is the functor

$$F(X)=\coprod_{\alpha \in Ob(\D)} F^\alpha_{X_\alpha} = \coprod_\alpha X_\alpha \otimes F_*^\alpha \mbox{ i.e. } F(X)_{\beta} = \coprod_\alpha \coprod_{\D(\alpha,\beta)} X_\alpha$$

For simplicity we will remain in the case of the generating trivial cofibrations. The proof for generating cofibrations is identical. A typical element of $J'$ is either $(0,j)$ or $(j,0)$ where $j:A\to A'$ is a generating trivial cofibration for $\M$. The generating trivial cofibrations of $\Arr(\M)$ are given by applying the functor $F$ to such generating morphisms. The resulting squares (read as morphisms from left to right) are

\begin{align*}
\xymatrix{\emptyset \ar[r] \ar[d] & \emptyset \ar[d] & & A \ar[r] \ar[d] & A' \ar[d]\\
A \ar[r] & A' & & A \ar[r] & A'}
\end{align*}

where the left-hand square is $\phi:=F(j,0)$ and the right hand square is $\psi:=F(0,j)$. We must analyze $\phi^{\boxprod_2 n}/\Sigma_n$ and $\psi^{\boxprod_2 n}/\Sigma_n$, where $\boxprod_2$ denotes the pushout product taken in the monoidal category $\Arr(\M)$ where the monoidal product on arrows is given by the Day convolution product (i.e. by the usual pushout product $\boxprod$). For simplicity we will focus on the case $n=2$, so we are analyzing $(\phi \boxprod_2 \phi)/\Sigma_2$ and $(\psi \boxprod_2 \psi)/\Sigma_2$. We can draw $\phi \boxprod_2 \phi$ and $\psi \boxprod_2 \psi$ as
\begin{align*}
\xymatrix{\emptyset \ar[r]^{Id} \ar[d] & \emptyset \ar[d] & & A'\otimes A \coprod_{A\otimes A} A\otimes A' \ar[r]_-{j\boxprod j} \ar[d] & A'\otimes A' \ar[d] \\
A'\otimes A \coprod_{A\otimes A} \limits A\otimes A' \ar[r]_-{j\boxprod j} & A'\otimes A' & & A'\otimes A \coprod_{A\otimes A} \limits A\otimes A' \ar[r]_-{j\boxprod j} & A'\otimes A'}
\end{align*}
In order to show that a square is a trivial cofibration in $Arr(\M)$ one must show that both horizontal maps in the square are trivial cofibrations in $\M$ and that the pushout corner map is a trivial cofibration in $\M$. 

For $(\phi \boxprod_2 \phi)/\Sigma_2$ the top horizontal map is $Id_\emptyset$ and so is a trivial cofibration. The bottom horizontal map and the pushout corner map are both $(j\boxprod j)/\Sigma_2$ and this is a trivial cofibration because we assumed the commutative monoid axiom on $\M$ (or the strong commutative monoid axiom for the case of cofibrations). Thus, $(\phi \boxprod_2 \phi)/\Sigma_2$ is a trivial cofibration in $\Arr(\M)$.

For $(\psi \boxprod_2 \psi)/\Sigma_2$ both the top and bottom horizontal maps are $(j\boxprod j)/\Sigma_2$ and we have seen this map is a trivial cofibration. The pushout corner map is $Id_{A'\otimes A'}/\Sigma_2 = Id_{A'\otimes A'/\Sigma_2}$ and so is a trivial cofibration in $\M$ because it's an isomorphism. Thus, $(\psi \boxprod_2 \psi)/\Sigma_2$ is a trivial cofibration in $\Arr(\M)$ and so $\Arr(\M)$ satisfies the $n=2$ case of the commutative monoid axiom.

Before tackling the case of general $n$, we record the general formula for 
\begin{align*}
\gamma \boxprod_2 \delta: f \boxprod k \coprod_{f\boxprod h} g\boxprod h \to g\boxprod k
\end{align*}
where $\gamma:f\to g$, $\delta:h\to k$, $f:A\to A'$, $g:B\to B'$, $h:X\to X'$, and $k:Y\to Y'$. Visualize $\gamma \boxprod_2 \delta$ as going from left to right:
\begin{align*}
\xymatrix{(A \otimes Y' \coprod_{A\otimes Y} \limits A'\otimes Y)
\coprod_{A \otimes X' \coprod_{A\otimes X}A'\otimes X} \limits
(B' \otimes X \coprod_{B\otimes X} \limits B\otimes X')
\ar[r] \ar[d] & 
B\otimes Y' \coprod_{B\otimes Y}\limits B'\otimes Y \ar[d] \\
A'\otimes Y' \coprod_{A' \otimes X'} \limits B'\otimes X' \ar[r] &
B'\otimes Y'}
\end{align*}

For the case of general $n$, observe that $\phi^{\boxprod_2 (n)}$ is again a square with $\emptyset$ across the top row, but now the bottom horizontal map is $j^{\boxprod n}$. This can be seen inductively, since $\phi^{\boxprod_2 (n)} = \phi^{\boxprod_2 (n-1)} \boxprod_2 \phi$ and the formula above tells us the effect of applying $-\boxprod_2 \phi$. We are lucky here since $\phi$ takes such a simple form. For $\psi$ the situation is slightly more complicated, but because both vertical maps in $\psi$ are identities the same will be true of $\psi^{\boxprod_2 (n)}$. A straight-forward induction demonstrates that the horizontal maps in $\psi^{\boxprod_2 (n)}$ are both $j^{\boxprod n}$. The same sort of analysis as in the $n=2$ case demonstrates that $\psi^{\boxprod_2 (n)}/\Sigma_n$ and $\phi^{\boxprod_2 (n)}/\Sigma_n$ are trivial cofibrations in $\Arr(\M)$ because $j^{\boxprod n}/\Sigma_n$ and identity maps are both trivial cofibrations in $\Arr(\M)$.

\end{proof}

As a consequence, the category of commutative Smith ideals inherits a model structure as soon as $\M$ satisfies the commutative monoid axiom. Using the results in Subsections \ref{subsubsec:posFlatSymSpec} and \ref{subsubsec:posSymSpec} we obtain the following:

\begin{corollary}
Suppose $\M$ is the positive flat (stable) model structure on symmetric spectra or orthogonal spectra. Then commutative Smith ideals inherit a model structure because $\M$ satisfies the strong commutative monoid axiom. Furthermore, a cofibrant commutative Smith ideal forgets to a cofibrant object of $Arr(\M)$.
\end{corollary}

\begin{corollary}
Suppose $\M$ is the positive (stable) model structure on symmetric spectra or orthogonal spectra. Then commutative Smith ideals inherit a model structure because $\M$ satisfies the commutative monoid axiom. Furthermore, a cofibrant commutative Smith ideal forgets to a positive flat cofibrant object of $Arr(\M)$.
\end{corollary}

In fact, we can use the method in the theorem above to prove something more general. First note that the commutative monoid axiom can be transferred across certain types of adjunctions. In the following, assume $F$ is a nonunital strongly symmetric monoidal functor in the sense of \cite{thomason-first-quadrant}. This simply means $F$ satisfies all the hypotheses of a strong symmetric monoidal functor except $F$ need not satisfy the hypotheses involving the unit. So $F(n_1 \otimes n_2)\cong F(n_1)\otimes F(n_2)$ and there are diagrams encoding associativity and commutativity for $F$, but $F$ need not preserve the unit.

\begin{lemma} \label{lemma:comm-mon-ax-to-proj-model}
Suppose $\cat{N}$ is a cofibrantly generated model category satisfying the (strong) commutative monoid axiom, that $F:\cat{N}\to \cat{A}$ is a left adjoint functor along which a model structure is transferred to $\cat{A}$, and that $F$ is a nonunital strongly symmetric monoidal functor. Then $\cat{A}$ satisfies the (strong) commutative monoid axiom.
\end{lemma}

\begin{proof}
As in the proof of the theorem above, it suffices to check that for every generating (trivial) cofibration $F(f)$ of $\cat{A}$, $(F(f))^{\boxprod n}/\Sigma_n$ is again a (trivial) cofibration. Because $F$ preserves the monoidal product and is left Quillen, $F$ commutes with the operation $(-)^{\boxprod n}$, so $(F(f))^{\boxprod n}/\Sigma_n = F(f^{\boxprod n})/\Sigma_n = F(f^{\boxprod n}/\Sigma_n)$ again using that $F$ is a left adjoint. Since $f$ is a generating (trivial) cofibration of $\cat{N}$, this map has the form $F(g)$ where $g$ is a (trivial) cofibration of $\cat{N}$. Hence, $F(g)$ is a (trivial) cofibration of $\cat{A}$ because $F$ is left Quillen.
\end{proof}

Let $\D$ be a general diagram category. The projective model structure on $\M^\D$ is transferred from $\M^{\D^{disc}} = \prod_{d\in \D} \M$ along the functor $F$ whose formula is given in the theorem above. If $\M^{\D^{disc}}$ is given the objectwise product then this functor $F$ does not preserve binary products in general. In the case $\Arr(\M)$ one can endow $\M^{\D^{disc}}$ with the monoidal structure $(X_1,Y_1)\boxprod_d (X_2,Y_2)=(X_1Y_2\coprod X_1Y_1 \coprod X_2Y_1,Y_1Y_2)$ which is cooked up to match the Day convolution product on $\Arr(\M)$, but then the functor from $\M$ to $\M^{\D^{disc}}$ may not be monoidal. In the following, we transfer the model structure from $\M$ rather than from $\M^{\D^{disc}}$.

Recall that for every $d\in \D$ there are right adjoints $ev_d:\M^{\D}_{proj}\to \M$ whose left adjoints are given by $F_d:X\mapsto \D(d,-)\cdot X = \coprod_{\D(d,-)}X$. The model structure on $\M^\D$ may be transferred directly from $\M$ via the product of these left adjoints (Lemma 4.3 in \cite{DRO}), so it is valuable to know when they are strongly symmetric monoidal.

\begin{theorem}
Suppose $(\M,S,\otimes)$ is a cofibrantly generated symmetric monoidal model category satisfying the (strong) commutative monoid axiom. Suppose $(\D,I,\circledast)$ is a small symmetric monoidal category in which the monoidal product is idempotent, i.e. there are natural isomorphisms $d\circledast d \cong d$ for all $d\in \D$ and they induce isomorphisms $\D(d\circledast d,-)\cong \D(d,-)$. Then the projective model structure on $\M^\D$ exists and the Day convolution product $\circ$ makes it into a cofibrantly generated symmetric monoidal model category satisfying the (strong) commutative monoid axiom.
\end{theorem}

\begin{proof}
The existence of a cofibrantly generated model structure on $\M^\D$ is an old result, which may be found as Theorem 11.6.1 in \cite{hirschhorn}. For the fact that the Day convolution product makes $\M^\D$ into a symmetric monoidal model category we refer to Theorem 4.1 in \cite{batanin-berger}. For the commutative monoid axiom we must show that the left adjoints $F_d$ preserve binary monoidal products. We must show that $F_d(A\otimes B) \cong F_d(A) \circ F_d(B)$. By definition of the Day convolution product, the latter can be realized as the following coend, which we re-write using the Yoneda lemma
\begin{align*}
\int^{x,y\in \D} (\D(d,x)\cdot A)\otimes (\D(d,y)\cdot B) \cdot \D(x\circledast y,-)\\
= \int^{x,y} (\D(d,x)\times \D(d,y) \times \D(x\circledast y,-))\cdot A\otimes B\\
= \D(d\circledast d,-)\cdot A\otimes B
\end{align*}
Our hypothesis on $\D$ guarantees that $\D(d\circledast d,-)\cdot A\otimes B = \D(d,-)\cdot A\otimes B$ as required. 
\end{proof}

Observe that the hypothesis that $F_d$ is strongly symmetric monoidal would have been too strong, since the only $d$ for which $F_d$ preserves the monoidal unit is $d=I$ the unit of $\D$.

In fact, Theorem 4.1 in \cite{batanin-berger} proves that $\M^\D$ is a symmetric monoidal model category with respect to the Day convolution product in a more general setting than that stated above. If one is willing to assume $\M$ is compactly generated then one can generalize to the case of a small $\M$-enriched category $\D$:

\begin{corollary}
Suppose $(\M,S,\otimes)$ is a compactly generated symmetric monoidal model category satisfying the (strong) commutative monoid axiom. Suppose $(\D,I,\circledast)$ is a small $\M$-enriched symmetric monoidal category in which the monoidal product is idempotent as above. Suppose one of the following three conditions
\begin{enumerate}
\item all hom-objects of $\D$ are discrete in $\M$ (this recovers the classical projective model structure), 
\item all hom-objects of $\D$ are cofibrant in $\M$, or
\item the monoid axiom holds in $\M$.
\end{enumerate}

Then the projective model structure on $\M^\D$ exists and the Day convolution product $\circ$ makes it into a cofibrantly generated symmetric monoidal model category satisfying the (strong) commutative monoid axiom.
\end{corollary}

We omit the proof since it is identical to the theorem above, except that we rely on Theorem 4.1 in \cite{batanin-berger} for the existence of the transferred model structure on $\M^\D$ as well as for the fact that it forms a monoidal model structure.

We also have a result in case $\M^\D$ is endowed with the objectwise monoidal product $(F\otimes G)(d) = F(d)\otimes G(d)$. In this case we may not rely on Theorem 4.1 in \cite{batanin-berger} and so we will focus attention again on a discrete indexing category $\D$.

\begin{corollary}
Suppose $\D$ is a small poset with finite coproducts and $(\M,S,\otimes)$ is a cofibrantly generated symmetric monoidal model category satisfying the (strong) commutative monoid axiom. Then the projective model structure on $\M^\D$ with the objectwise monoidal product $\otimes$ is a cofibrantly generated symmetric monoidal model category satisfying the commutative monoid axiom.
\end{corollary}

\begin{proof}
That the projective model structure exists is again described in Theorem 11.6.1 of \cite{hirschhorn}. That it satisfies the pushout product axiom can be found in Lemma 3.8 of \cite{yalin-alg-over-prop}, because we have assumed $\D$ has finite coproducts. We must now verify that $F_d(A\otimes B)\cong F_d(A) \otimes F_d(B)$ for all $d\in \D$. The latter is now $(\D(d,-) \times \D(d,-))\cdot A\otimes B$. It is an easy exercise to verify that for any poset $\D$, this is isomorphic to $\D(d,-)\cdot A\otimes B$ as required.
\end{proof}

Relating the situations of the above two results we have the following

\begin{proposition}
If $\D$ is a Cartesian closed small category then the Day convolution product agrees with the objectwise product on $\M^{\D^{op}}$.
\end{proposition}

\begin{proof}
Let $X,Y:\D^{op} \to \M$. Let $\times$ denote the Cartesian product on $\D$ and on $\M^{\D^{op}}$. We use the Cartesian assumption in (*) below and we use the closed hypothesis for the next equality after (*):

$$(X \circ Y)_\beta = \int^{\gamma,\delta \in \D} \coprod_{\D(\beta,\gamma \times \delta)} X_\gamma \otimes Y_\delta$$
$$ \stackrel{*}{=} 
\int \coprod_{\stackrel{\beta\to \gamma}{\beta\to \delta}} X_\gamma \otimes Y_\delta = 
\int^\gamma \coprod_{\beta \to \gamma}X_\gamma \otimes \int^\delta \coprod_{\beta\to\delta}Y_\delta = X_\beta \otimes Y_\beta = (X\times Y)_\beta$$

The last non-trivial equality is the Yoneda lemma.
\end{proof}

\subsubsection{Reedy Model Structures}

We have focused on the projective model structure, but for directed categories analogous statements hold for the Reedy model structure, since it coincides with the projective model structure. In order to make the Reedy structure into a monoidal model category we require two results of Barwick \cite{barwickSemi}. In the following, assume $\D$ is a Reedy category and $\M$ is a monoidal model category.

\begin{proposition}[Barwick's Corollary 4.17]
Suppose $(\D,I,\circledast)$ is monoidal and $\circledast:\D \times \D \to \D$ defines a right fibration of Reedy categories, i.e. if for any model category $\cat{N}$, the induced adjunction $\cat{N}^\D \leftrightarrows \cat{N}^{\D \times \D}$ is a Quillen adjunction. Then $\M^\D$ with the Reedy model structure and the Day convolution product satisfies the pushout product axiom.
\end{proposition}

\begin{proposition}[Barwick's Theorem 4.18]
Suppose $\D$ is left fibrant, i.e. for any model category $\cat{N}$ the Reedy functor $\cat{N}\to \ast$ where $\ast$ is the terminal Reedy category induces an adjunction $\cat{N}^\D \leftrightarrows \cat{N}^\ast$ which is a Quillen adjunction. Suppose that every morphism in the inverse category associated to $\D$ is an epimorphism. Then the objectwise monoidal product endows the Reedy model category $\M^\D$ with the structure of a monoidal model category.
\end{proposition}

From these, together with our results on projective model structures and the fact that Reedy model structures coincide with projective model structures whenever $\D$ is a directed category we have:

\begin{corollary}
Suppose $(\D,I,\circledast)$ is a small directed monoidal category, that $\circledast:\D \times \D \to \D$ defines a right fibration of Reedy categories, and that $d\circledast d \cong d$ for all $d\in \D$. Suppose $\M$ is a cofibrantly generated monoidal model category satisfying the (strong) commutative monoid axiom. Then $\M^\D$ with the Reedy model structure and the Day convolution product satisfies the (strong) commutative monoid axiom.
\end{corollary}

\begin{corollary}
Suppose $\D$ is a small directed poset with finite coproducts and is left fibrant as a Reedy category. Suppose $\M$ satisfies the (strong) commutative monoid axiom. Then the Reedy model structure on $\M^\D$ with the objectwise monoidal product satisfies the (strong) commutative monoid axiom.
\end{corollary}

In the second corollary we were able to remove Barwick's condition about epimorphisms because our category $\D$ has a trivial inverse category.

\subsection{Other Examples}

There are several other examples which we have not investigated and which we would be curious to learn more about. We list them here:

\begin{itemize}
\item Stable module categories.

\item Comodules over a Hopf algebroid.

\item The model for spectra consisting of simplicial functors, in the style of \cite{lydakis}.

\end{itemize}

We have not addressed positive model structures on motivic symmetric spectra. We understand that these examples are central to the work of \cite{dmitri}, which will appear soon.

\appendix
\section{Sufficiency of Commutative Monoid Axiom on Generators}
\label{appendix-generators}

We prove that if the strong commutative monoid axiom holds for the generating (trivial) cofibrations $I$ and $J$ then it holds for all (trivial) cofibrations.

\begin{lemma}
Suppose $\M$ is a cofibrantly generated monoidal model category and that for all $f \in I$ (resp. $J$) we know that $f^{\boxprod n}/\Sigma_n$ is a (trivial) cofibration. Then the strong commutative monoid axiom holds for $\M$.
\end{lemma}

We will prove that the class of maps satisfying the condition in the strong commutative monoid axiom is closed under retracts, pushouts, and transfinite compositions. The first two are easy, but the third will require an induction. So we must introduce some new notation, following \cite{harper-operads}. Let $f:X\to Y$ and consider the $n$-dimensional cube in which each vertex is a word of length $n$ on the letters $X$ and $Y$. 

Recall the action of $\Sigma_n$ on the diagram which defines $Q_n$. The vertices of the cube correspond to subsets $D$ of $[n]=\{1,2,\dots,n\}$ where a vertex $C_1\otimes \dots \otimes C_n$ has $C_i = X$ if $i\notin D$ and $C_j = Y$ if $j\in D$. Any $\sigma \in \Sigma_n$ sends the vertex so defined to the vertex corresponding to $\sigma(D) \subset [n]$ using the action of $\Sigma_{|D|}$ on the $X$'s and $\Sigma_{n-|D|}$ on the $Y$'s. Clearly, this action descends to an action on the colimit $Q_n$.

For inductive purposes, we will need to consider subdiagrams whose vertices consist of words with $\leq q$ copies of the letter $Y$. This subdiagram consists of all vertices of distance $\leq q$ from the initial vertex $X^{\otimes n}$. We denote the colimit of this subdiagram by $Q_q^n$, to match the notation of \cite{harper-operads}. The superscript $n$ refers to the fact that this is a subdiagram of the $n$-dimensional cube, so in particular each vertex is a word on $n$ letters.  In particular, $Q_0^n = X^{\otimes n}$, $Q_n^n = Y^{\otimes n}$, and $Q_{n-1}^n$ is the domain of $f^{\boxprod n}$, which we have formerly denoted by $Q_n$. For the purposes of this proof we will now write it as $Q_{n-1}^n(f)$ (or $Q_{n-1}^n$ if the context is clear). 

The induction will make use of the maps of colimits $Q_{q-1}^n \to Q_q^n$ which are induced by inclusion of subdiagram. The $\Sigma_n$ action on the cube clearly preserves the size of the subset $D\subset [n]$ and so it restricts to an action of $\Sigma_n$ on each $Q_q^n$. Because this action is a restriction of the $\Sigma_n$-action on the full cube, the map of colimits $Q_{q-1}^n \to Q_q^n$ is automatically $\Sigma_n$-equivariant. Indeed, the map of colimits $Q_{q-1}^n \to Q_q^n$ can be realized by the following pushout:
\begin{align}
\label{eq:pushout_defining_equivariant_punctured_cube}
\xymatrix{
  \Sigma_n \cdot_{\Sigma_{n-q}\times\Sigma_{q}} X^{\otimes (n-q)} 
  \otimes Q_{q-1}^q \ar[d] \ar[r] \po & Q_{q-1}^n\ar[d]  \\
  \Sigma_n\cdot_{\Sigma_{n-q}\times\Sigma_{q}}X^{\otimes(n-q)}
  \otimes Y^{\otimes q}\ar[r] & Q_q^n
}
\end{align}

where the left vertical map is induced by $f^{\boxprod q}$ (see Section 7 of \cite{harper-operads} and Remark 4.15 of \cite{harper-spectra} for a toy case). 
To explain the notation $\Sigma_n \cdot_{\Sigma_{n-q}\times\Sigma_{q}} (-)$, first note that for any set $G$ and any object $A$, $G\cdot A = \coprod_{g\in G} A$. When $G=\Sigma_n$ this object inherits a $\Sigma_n$ action by permuting the $A^{n!}$ objects in the coproduct. When we write $\Sigma_n \cdot_{\Sigma_k \times \Sigma_q}(-)$ we are quotienting out by the $\Sigma_k \times \Sigma_q$ action on this object in $\M^{\Sigma_n}$. The result is a coproduct with $n! / (k!q!)$ terms because the order of the $k!$ terms to the left of the product (and of the $q!$ terms to the right) do not matter. In particular, applying $\Sigma_n \cdot_{\Sigma_k \times \Sigma_q}(-)$ has the effect of \textit{equivariantly} building in additional layers of the cube. With this notation in hand we proceed to the proof.

\begin{proof}
Let $\sP$ denote the class of cofibrations $f$ for which $f^{\boxprod n}/\Sigma_n$ is also a cofibration. Let $\sP'$ denote the same for trivial cofibrations. We must prove that if $I \subset \sP$ then all cofibrations are in $\sP$ (and the same for $J\subset \sP'$). We will do so by proving the classes $\sP$ and $\sP'$ are closed under retracts, pushouts, and transfinite compositions.

The simplest to verify is closure under retracts, which follows from the fact that $(-)^{\boxprod n}/\Sigma_n$ is a functor on $Arr(\M)$ so if $f$ is a retract of $g$ (with $g\in \sP$ or $\sP'$) then $f^{\boxprod n}/\Sigma_n$ is a retract of $g^{\boxprod n}/\Sigma_n$ and hence a (trivial) cofibration.

We next consider closure under pushouts. Suppose $f:X\to Y$ is a pushout of $g:A\to B$ and $g\in \sP$ or $\sP'$. Then we have a $\Sigma_n$-equivariant pushout diagram
\begin{align*}
\xymatrix{Q_n(g) \ar[r] \ar[d] \po & B^{\otimes n} \ar[d] \\ Q_n(f) \ar[r] & Y^{\otimes n}}
\end{align*}

by Proposition 6.13 in \cite{harper-spectra}. When we pass to $\Sigma_n$-coinvariants we see that $f^{\boxprod n}/\Sigma_n$ is a pushout of $g^{\boxprod n}/\Sigma_n$, e.g. by commuting colimits.
Indeed, for any $X\in \M^{\Sigma_n}$, $X\otimes_{\Sigma_n} f^{\boxprod n}$ is a pushout of $X\otimes_{\Sigma_n} g^{\boxprod n}$. So if the latter is assumed to be a (trivial) cofibration because $g \in \sP$ or $\sP'$ then the former will be as well.

Composition is harder, so we begin with the case of two maps $f:X\to Y$ and $g:Y\to Z$ in $\sP$ or $\sP'$. We will prove that $Q^n_{n-1}(gf)/\Sigma_n \to Z^{\otimes n}/\Sigma_n$ is a (trivial) cofibration. First note that this map factors through $Q^n_{n-1}(g)/\Sigma_n$ and the hypothesis on $g$ guarantees that $Q^n_{n-1}(g)/\Sigma_n \to Z^{\otimes n}/\Sigma_n$ is a (trivial) cofibration. So we must only prove that $Q^n_{n-1}(gf)/\Sigma_n \to Q^n_{n-1}(g)/\Sigma_n$ is a (trivial) cofibration.

We proceed by realizing both colimit diagrams as subdiagrams of the same diagram, which is a $n$-dimensional cube featuring $3^n$ vertices which are words of length $n$ in the letters $X, Y,$ and $Z$. Formally, this cube is an element of the rectangular diagram category $Fun((0 \to 1 \to 2)^{\times n}, \M)$, and every time we write subdiagram we mean with respect to this cube with $3^n$ vertices. The domain $Q^n_{n-1}(gf)$ of the map we care about is the colimit of the $X-Z$ subdiagram, i.e. the punctured cube formed from vertices which are words in $X$ and $Z$, where all maps are compositions $gf$. The codomain $Q^n_{n-1}(g)$ of the map we care about is the colimit of the $Y-Z$ subdiagram, i.e. the punctured cube formed from vertices which are words in $Y$ and $Z$. So we must again introduce new notation to build this map one step at a time. 

The induction will proceed by moving through the rectangle by adding a single $\Sigma_n$-orbit at a time. So we will need to consider $\Sigma_n$-equivariant subdiagrams of the rectangle which contain the $X-Z$ punctured cube and which contain a new vertex $e$ (and hence its entire $\Sigma_n$-orbit). 

In order to build this new vertex into the colimit we will also need to consider the subdiagram of the $X-Y-Z$ box which maps to $e$ (but which does not include $e$ itself). This is collection of vertices sitting under $e$ (i.e. of distance strictly less than $e$ from the initial vertex). As with $e$, we wish to consider the $\Sigma_n$-orbit of this subdiagram, which is equivalently described as all vertices sitting under any vertex in the orbit of $e$. Now that we have a picture of the subdiagram in mind, we denote the colimit of this subdiagram by $Q_e$. By construction there is an induced $\Sigma_n$-equivariant map $Q_e \to e$.

We are now ready to consider the diagrams formed when we adjoin the $Q_e$-diagram with the $X-Z$ punctured cube. Let $Q[0]^n_{n-1}=Q^n_{n-1}(gf)$ denote the colimit of the $X-Z$ punctured cube. Let $Q[1]^n_{n-1}$ denote the colimit of the subdiagram containing the $X-Z$ punctured cube, the orbit of the vertex $e = Y\otimes Z^{\otimes (n-1)}$, and the vertices in the $Q_e$ subdiagram. Continue inductively, by adding $e = Y^{\otimes q}\otimes Z^{\otimes (n-q)}$ and vertices below it to the $Q[q-1]^n_{n-1}$-diagram to get the $Q[q]^n_{n-1}$-diagram. 
This process terminates with the whole $X-Y-Z$ punctured cube whose $3^n-1$ vertices contain all words in $X,Y,Z$ except the word $Z^{\otimes n}$. The colimit of this diagram is denoted $Q[n]^n_{n-1}$. A cofinality argument shows that this colimit is equal to $Q^n_{n-1}(g)$, because all factors of $X$ which appear are mapped to a factor of $Y$ in the subdiagram and so do not affect the colimit.

The induction will proceed along the maps $Q[q-1]^n_{n-1} \to Q[q]^n_{n-1}$ induced by containments of subdiagrams. This induction can be thought of as stepping through shells in the cube of increasing distance from the initial vertex $X^{\otimes n}$ until the information from the entire diagram has been built into the colimit.

Because each step $Q[q-1]^n_{n-1} \to Q[q]^n_{n-1}$ builds in the information of one new vertex (and its orbit under the $\Sigma_n$ action on the cube), we may apply Proposition A.4 from \cite{luis} with $e=Y^{\otimes q} \otimes Z^{\otimes (n-q)}$ to write the following pushout diagram:
\begin{align} \label{eq:defining-for-Q[q]-to-Q[q+1]}
\xymatrix{
  \Sigma_n \cdot_{\Sigma_{q}\times\Sigma_{n-q}} Q_{e} \ar[d] \ar[r] \po & Q[q]_{n-1}^n \ar[d] \\
  \Sigma_n\cdot_{\Sigma_{q}\times\Sigma_{n-q}}Y^{\otimes q}
  \otimes Z^{\otimes (n-q)}\ar[r] & Q[q+1]_{n-1}^n
}
\end{align}

The left vertical map is induced by $Q_{e} \to Y^{\otimes q} \otimes Z^{\otimes (n-q)}$ and this is in turn induced by $f^{\boxprod q} \boxprod g^{\boxprod (n-q)}$ because 
\begin{align}
\label{eq:cube-decomposition-for-composition}
Q_e \cong Y^{\otimes q} \otimes Q^{n-q}_{n-q-1}(g) \coprod_{Q_{q-1}^q(f) \otimes Q_{n-q-1}^{n-q}(g)} Q_{q-1}^q(f) \otimes Z^{\otimes (n-q)}
\end{align}

To see that the diagram defining $Q_e$ decomposes into a gluing of the diagrams defining $Q_{q-1}^q(f)\otimes Z^{n-q}$ and $Y^q \otimes Q_{n-q-1}^{n-q}(g)$ along the diagram defining $Q_{q-1}^q(f) \otimes Q_{n-q-1}^{n-q}(g)$, note that every $X$ in the $Q_e$ diagram gets mapped to a $Y$ in the $Q_e$ diagram and so does not affect the colimit. This is the reason why we insisted upon including the vertices under $e$ in our construction of the diagram defining $Q_e$. Furthermore, every $Z$ in the $Q_e$ diagram is the image of some $Y$ and so we may apply a cofinality argument to realize that any map out of the diagram for the left-hand side of (\ref{eq:cube-decomposition-for-composition}) must factor through the right-hand side, which completes the proof of (\ref{eq:cube-decomposition-for-composition}).

Now pass to $\Sigma_n$-coinvariants in Equation (\ref{eq:defining-for-Q[q]-to-Q[q+1]}). Verifying that the left vertical map is a cofibration reduces to verifying that $f^{\boxprod q}/\Sigma_q \boxprod g^{\boxprod (n-q)}/\Sigma_{n-q}$ is a cofibration. This in turn follows from the inductive hypothesis on $f$ and $g$. Thus all the maps $Q[q]^n_{n-1}/\Sigma_n \to Q[q+1]^n_{n-1}/\Sigma_n$ are pushouts of cofibrations and hence are cofibrations themselves. Hence, their composite $Q^n_{n-1}(gf)/\Sigma_n \to Q^n_{n-1}(g)/\Sigma_n$ is a cofibration. This completes the proof that the classes $\sP$ and $\sP'$ are closed under composition.

Finally, we cover the case of transfinite composition. First note that the proof for the composition of two maps proves that the vertical maps 
and the induced pushout corner map in the following square 
become cofibrations after passing to $\Sigma_n$-coinvariants, by the general machinery of adding a new vertex $e$ containing only $Y$s and $Z$s:
\begin{align*}
\xymatrix{Q_{t-1}^n(f) \ar[r] \ar[d] & Q_{t-1}^n(gf) \ar[d] \\ Q^n_t(f) \ar[r] & Q^n_t(gf)}
\end{align*}

Indeed, the same is true of the diagram 
\begin{align*}
\xymatrix{Q^n_{n-1}(f) \ar[r] \ar[d] & Q_{n-1}^n(gf) \ar[d] \\ Y^{\otimes n} \ar[r] & Z^{\otimes n}}
\end{align*}

This is the analogous result to Corollary A.7 in \cite{luis}, which begins with power cofibrations and concludes that the diagram represents a projective cofibration in Arr$(\M^{\Sigma_n})$. Recall, e.g. from Definition 2.1 in \cite{obstruction} that a square is a projective cofibration if and only if the vertical maps and the pushout corner map are cofibrations. In our situation we pass to $\Sigma_n$-coinvariants on the diagram level and in that way achieve a projective cofibration in Arr$(\M)$

Now let $X_0\stackrel{f_0}{\to}, X_1 \stackrel{f_1}{\to}, X_2 \stackrel{f_2}{\to}, \dots$ be a $\lambda$-sequence in which each $f_\alpha \in \sP$. Let $f_\infty:X_0\to X_\lambda$ be the composite. To prove that $f_\infty ^{\boxprod n}/\Sigma_n$ is a cofibration, we realize this map as the colimit of a particular diagram. Because colimits commute we can pass to $\Sigma_n$-coinvariants in the diagram and we will see that the colimit of the resulting diagram (which will be $f_\infty ^{\boxprod n}/\Sigma_n$) will be a cofibration. First we realize the domain of $f_\infty ^{\boxprod n}/\Sigma_n$ as a colimit along the sequence $Q^n_{n-1}(f_0) \to Q^n_{n-1}(f_1f_0) \to Q^n_{n-1}(f_2f_1f_0) \to \dots Q^n_{n-1}(f_\infty)$. Next, we realize $f_\infty ^{\boxprod n}$ as the far right-hand map in
\begin{align}
\label{eq:Qn-of-transfinite-composition}
\xymatrix{Q^n_{n-1}(f_0) \ar[r] \ar[d] & Q^n_{n-1}(f_1f_0) \ar[r] \ar[d] & Q^n_{n-1}(f_2f_1f_0)\ar[r] \ar[d] & \dots \ar[r] & Q^n_{n-1}(f_\infty) \ar[d] \\
X_0^{\otimes n} \ar[r] & X_1^{\otimes n} \ar[r] & X_2^{\otimes n} \ar[r] & \dots \ar[r] & X_\lambda^{\otimes n}}
\end{align}

As in the case for two-fold composition, we pass to $\Sigma_n$-coinvariants in this diagram and realize that the resulting diagram is a projective cofibration in the category of $\lambda$-sequences $\M^\lambda$ because all vertical maps and all pushout corner maps are cofibrations. The colimit of such a diagram must be a cofibration, because colimit is a left Quillen functor from $\M^\lambda \to \M$. This proves that $f_\infty ^{\boxprod n}/\Sigma_n$ is a (trivial) cofibration as desired.

\end{proof}

\begin{remark}
The author is indebted to Luis Pereira for many helpful conversations as this proof was worked out. The author's original proof proceeded by constructing a lift to prove that $Q^n_{n-1}(gf)/\Sigma_n \to Q^n_{n-1}(g)/\Sigma_n$ has the left lifting property with respect to all (trivial) fibrations. This proof comes down to constructing an equivariant lift at the level of the cube diagrams, 
and it appears something similar has been done by \cite{gorchinskiy-symmetrizable},
though we find the proof presented here conceptually simpler. In \cite{luis}, Pereira uses a similar proof to prove that it is sufficient to check Jacob Lurie's axiom on the generating (trivial) cofibrations, at least in the case when the domain $X$ of $f$ is cofibrant. Pereira in fact proves something more general about the intermediate maps $Q[q-1]\to Q[q]$. The proof presented here avoids the need for $X$ to be cofibrant, even in Pereira's situation of working with Lurie's axiom rather than the strong commutative monoid axiom.
\end{remark}

We conclude this appendix by recording a result related to the proof above which is used in Section \ref{subsec:left-proper} and in the companion paper \cite{white-localization}:

\begin{lemma} \label{lemma:boxprod-equiv-if-sym}
Assume that for every $g\in I$, $g^{\boxprod n}/\Sigma_n$ is a cofibration. Suppose $f:X\to Y$ is a trivial cofibration between cofibrant objects and $f^{\boxprod n}/\Sigma_n$ is a cofibration for all $n$. Then $f^{\boxprod n}/\Sigma_n$ is a trivial cofibration for all $n$ if and only if $\Sym^n(f)$ is a trivial cofibration for all $n$.
\end{lemma}

\begin{proof}
We have seen above that if all maps $g$ in $I$ (resp. $J$) have the property that $g^{\boxprod n}/\Sigma_n$ is a (trivial) cofibration, then the same holds for all (trivial) cofibrations $g$. Thus, the hypothesis implies that the class of cofibrations is closed under the operation $(-)^{\boxprod n}/\Sigma_n$. 

Recall the construction of the $\Sigma_n$-equivariant maps $Q_{q-1}^n \to Q_q^n$ from our proof above:
\begin{align}
\label{eq:pushout_defining_equivariant_punctured_cube}
\xymatrix{
  \Sigma_n \cdot_{\Sigma_{n-q}\times\Sigma_{q}} X^{\otimes (n-q)} 
  \otimes Q_{q-1}^q \ar[d] \ar[r] \po & Q_{q-1}^n\ar[d]  \\
  \Sigma_n\cdot_{\Sigma_{n-q}\times\Sigma_{q}}X^{\otimes(n-q)}
  \otimes Y^{\otimes q}\ar[r] & Q_q^n
}
\end{align}

Observe that the pushout diagram above remains a pushout diagram if we apply $(-)/\Sigma_n$ to all objects and morphisms in the diagram, because $(-)/\Sigma_n$ is a left adjoint and so commutes with colimits. We obtain the diagram
\begin{align}
\label{eq:pushout_defining_equivariant_after_coinvariants}
\xymatrix{
  \Sym^{n-q} (X) \otimes Q_{q-1}^q /\Sigma_q \ar[d] \ar[r] \po & Q_{q-1}^n /\Sigma_n \ar[d]  \\
  \Sym^{n-q}(X) \otimes \Sym^q(Y)\ar[r] & Q_q^n /\Sigma_n
}
\end{align}

We have assumed $X$ is cofibrant, so $\Sym^k(X)$ is cofibrant for all $k$ by the hypothesis of the lemma applied to the cofibration $g:\emptyset \to \Sym^k(X)$. Thus, the left vertical map above is a trivial cofibration as soon as $f^{\boxprod q}$ is a trivial cofibration, by the pushout product axiom. 

We are now ready to prove the forwards direction in the lemma. Fix $n$ and realize $\Sym^n(f)$ as a composite of maps $Q_{q-1}^n/\Sigma_n \to Q_q^n/\Sigma_n$ as above, where $Q_0^n = X^{\otimes n}$ and $Q_n^n = Y^{\otimes n}$. Assume $f^{\boxprod q}$ is a trivial cofibration for all $q$ and deduce that each $Q_{q-1}^n/\Sigma_n \to Q_q^n/\Sigma_n$ is a trivial cofibration, because trivial cofibrations are closed under pushout. Furthermore, because trivial cofibrations are closed under composite, this proves $\Sym^n(f)$ is a trivial cofibration.

To prove the converse, assume that $\Sym^k(f)$ is a trivial cofibration for all $k$. We will prove $f^{\boxprod n}/\Sigma_n$ is a trivial cofibration for all $n$ by induction. For $n=1$ the map is $f$, which we have assumed to be a trivial cofibration. Now assume $f^{\boxprod i}/\Sigma_i$ is a trivial cofibration for all $i<n$. As in the proof of Lemma \ref{lemma:generators-suffice} we may again prove $f^{\boxprod n}/\Sigma_n$ is a trivial cofibration via the filtration in diagram (\ref{eq:pushout_defining_equivariant_after_coinvariants}). By our inductive hypothesis, we know that for all $i<n$, $Q^n_{i-1}/\Sigma_n \to Q^n_i /\Sigma_n$ is a trivial cofibration. We therefore have a composite:
\begin{align*}
\Sym^n(X)=Q^n_0/\Sigma_n \to Q^n_1/\Sigma_n \to \dots \to Q^n_{n-1}/\Sigma_n \to Q^n_n/\Sigma_n = \Sym^n(Y)
\end{align*}

in which each map except the last is a trivial cofibration. However, we have assumed $\Sym^n(X)\to \Sym^n(Y)$ is a trivial cofibration, so by the two out of three property the map $Q^n_{n-1}/\Sigma_n \to Q^n_n/\Sigma_n$ is in fact a weak equivalence. This map is $f^{\boxprod n}/\Sigma_n$, and is a cofibration by hypothesis, so it is a trivial cofibration. This completes the induction.

\end{proof}

\section{Proof of Main Theorem} \label{appendix-proof}

As described in Section \ref{sec:main}, it is sufficient to prove the statements of Theorem \ref{thm:commMonModel} and Proposition \ref{prop:cofibrations-forget} for the case $R=S$ of commutative monoids in $\M$. Before proceeding to the proof, we fix some notation. Given a map $g:K\to L$ one can form $g^{\boxprod n}:\overline{Q_n}\to L^{\otimes n}$. This map is a (trivial) cofibration if $g$ is such, by the pushout product axiom. The domain and codomain both have an action of $\Sigma_n$. Modding out by this action gives a map which is denoted by $f^{\boxprod n}/\Sigma_n:Q_n/\Sigma_n \to \Sym^n(L) = L^{\otimes n}/\Sigma_n$.

The proofs of Theorem \ref{thm:commMonModel} and Proposition \ref{prop:cofibrations-forget} follow the proof in \cite{SS00} that $\Mon(\M)$ has a model structure inherited from $\M$. Because that proof is based on the general theory of monads (c.f. Lemma 2.3) it will go through verbatim if Lemma 6.2 in \cite{SS00} can be generalized to describe pushouts in $\CMon(\M)$ rather than in $\Mon(\M)$. We state the analogue to Lemma 6.2:

\begin{lemma} \label{lemma:appendix-main}
\begin{enumerate}
\item If $\M$ satisfies the commutative monoid axiom then in the category $\CMon(\M)$, $\Sym(J)$-cell is contained in the collection of maps of the form $(id_Z\otimes J)$-cell in $\M$. If in addition $\M$ satisfies the monoid axiom then these maps are weak equivalences in $\M$ and hence in $\CMon(\M)$.

\item If $\M$ satisfies the strong commutative monoid axiom then maps in $\Sym(I)-cell$ with cofibrant domain (in $\M$) are cofibrations in $\M$.

\end{enumerate}
\end{lemma}

As in \cite{SS00}, the proof of this proposition requires a careful analysis of the filtration on pushouts in the category of commutative monoids. In particular, we must prove the following.

\begin{proposition} \label{prop:comm-mon-pushout}
Given any map $h:K\to L$ in $\M$, consider the commutative monoid homomorphism $X\to P$ formed by the following pushout in $\CMon(\M)$
\begin{align*}
\xymatrix{\Sym(K) \ar[r] \po \ar[d] & \Sym(L) \ar[d] \\ X \ar[r] & P}
\end{align*}

This map $X\to P$ factors as $X=P_0\to P_1\to \dots \to P$, where each $P_{n-1}\to P_n$ is defined by the following pushout in $\M$
\begin{align*}
\xymatrix{X\otimes Q_n/\Sigma_n \ar[r] \po \ar[d] & X\otimes \Sym^n(L) \ar[d] \\ P_{n-1} \ar[r] & P_n}
\end{align*}

Here $Q_n = Q^n_{n-1}$ denotes the colimit of the $n$-dimensional punctured cube discussed in Appendix \ref{appendix-generators}.
\end{proposition}

This filtration is analogous to the one given in \cite{SS00}, and makes use of the decomposition $\Sym(-) = \oplus \Sym^n(-)$. The map $g:K\to X$ needed for the construction of $P_{n-1}\to P_n$ is adjoint to the map $\Sym(K)\to X$. Note that this description of $P_n$ is significantly simpler than the one found in \cite{SS00} because commutativity means one need not consider words with $X$'s, $K$'s, and $L$'s interspersed. Rather, all the $X$'s can be shuffled to the left and multiplied at the beginning of the process, rather than at the end as is done in \cite{SS00}. It is for this reason that Theorem \ref{thm:left-proper} only requires the hypothesis that $\M$ be $h$-monoidal rather than strongly $h$-monoidal as is required for monoids. If we were to keep our notation in line with the notation in \cite{SS00} then what we call $Q_n$ would be denoted $\overline{Q_n}$, but we will avoid this unnecessary shift in notation, because we will have no need for colimits of cubes formed from words in the letters $X,K,L$.

Once we prove this proposition, we will restrict attention to the case when $h=j$ is a trivial cofibration to prove the first statement in Lemma \ref{lemma:appendix-main} and we will restrict to when $h=i$ is a cofibration and $X$ is cofibrant for the second statement. This is done at the end of the section.

\begin{proof}[Proof of Proposition \ref{prop:comm-mon-pushout}]

We begin by describing the left vertical map in the diagram which defines the maps $P_{n-1}\to P_n$. This will be done inductively. Because $X\otimes -$ commutes with colimits (since it's a left-adjoint), the map $X\otimes Q_n/\Sigma_n \to P_{n-1}$ may be defined componentwise on the vertices of the cube defining $X\otimes Q_n$.

For the $n=1$ case the map $X\otimes K\to X\otimes X \to X=P_0$ is $g$ followed by $\mu_X:X\otimes X\to X$. Let $D$ be a proper subset of $[n]=\{1,\dots,n\}$ and define $W(D) = C_1\otimes \dots \otimes C_n$ where $C_i=K$ if $i\not \in D$ and $C_i=L$ if $i\in D$. These are the vertices of the cube defining $Q_n$. Given a vertex $X\otimes W(D)$ define a map by first applying $g$ to all factors of $K$ (call this map $g^*$), then shuffling all the factors of $X$ so obtained to the left by a permutation $\sigma_D$, then multiplying these factors together. This map takes $X\otimes W(D)$ to $X\otimes L^{\otimes |D|}$ and hence to $X\otimes \Sym^{|D|}(L)$ by passing to $\Sigma_n$-coinvariants. Induction then gives a map to $P_{|D|}$ and hence to $P_{n-1}$ because $D$ was a proper subset of $[n]$.

The map above is well-defined (i.e. respects the $\Sigma_n$ action on the cube defining $X\otimes Q_n$) because a permutation $\sigma$ which takes $W(D)$ to a different vertex $W(T)$ for some $T$ of the same size as $D$ yields the following commutative diagram:
\begin{align*}
\xymatrix{X\otimes W(D) \ar[r] \ar[dd]^{1\otimes \sigma} & X^{\otimes (n-|D|)} \otimes L^{\otimes |D|} \ar[r] \ar@{..>}[dd] & X\otimes L^{\otimes |D|} \ar[dr]& \\
& & & X\otimes L^{\otimes |D|}/\Sigma_n\\
X\otimes W(T) \ar[r] & X^{\otimes (n-|D|)}\otimes L^{\otimes |D|} \ar[r] & X\otimes L^{\otimes |D|} \ar[ur] & }
\end{align*}

The left square commutes because the top left horizontal map is $\sigma_D \circ g^*$ and the bottom left horizontal map is $\sigma_T \circ g^*$, so the dotted arrow can be defined as $\sigma|_{D}$ on the $|D|$ factors of $L$ and as $\sigma|_{[n]-D}$ on the $n-|D|$ factors of $X$ (using the fact that $X$ is commutative). Thus, both ways of going around are simply doing $g^*, \sigma,$ and the shuffling of $X$'s to the left. The right pentagon commutes because $X$ is commutative (so the order of factors doesn't matter) and because passage to $\Sigma$-coinvariants means the order of factors of $L$ does not matter either.

These maps from vertices assemble to a map from $X\otimes Q_n \to P_{n-1}$ because taking $i\not \in D$ and defining the map from $X\otimes W(D\cup \{i\})\to P_{n-1}$ as above gives a diagram, which we will show commutes:
\begin{align*}
\xymatrix{X\otimes W(D) \ar[r] \ar[d] & X\otimes L^{\otimes |D|} \ar[r] & P_{|D|}\ar[d]\\
X\otimes W(D\cup \{i\}) \ar[r] & X\otimes L^{\otimes (|D|+1)} \ar[r] & P_{|D|+1}}
\end{align*}

The upper left horizontal map is $\mu_X \circ \sigma_D \circ g^*$ so we may factor it as $X\otimes W(D) \to X^{\otimes (n-|D|-1)}\otimes K\otimes L^{\otimes |D|} \to X\otimes L^{\otimes |D|}$ where $K$ is the $i^{th}$ factor of the original $W(D)$. Since this factor becomes an $L$ in the bottom row we have the following diagram:
\begin{align*}
\xymatrix{X\otimes W(D) \ar[r] \ar[d] & X\otimes K\otimes L^{\otimes |D|} \ar[r] \ar[d] & P_{|D|}\ar[d]\\
X\otimes W(D\cup \{i\}) \ar[r] & X\otimes L^{\otimes (|D|+1)} \ar[r] & P_{|D|+1}}
\end{align*}

The difference between the two ways of going around the left-hand square is the order of factors in the $L$ component (the order in the $X$ component doesn't matter). Thus, this square will commute upon passage to $P_{|D|+1}$ because of passage to $\Sigma_{|D|+1}$-coinvariants. Recall that $P_{|D|+1}$ is a pushout of $X\otimes Q_{|D|+1}$, which is itself a pushout of vertices $X\otimes W(R)$. Because a pushout of a pushout is again a pushout, the right-hand square commutes by the basic property of pushouts.

This completes the inductive definition of $P_n$. Setting $P$ to be the colimit of the $P_n$ (taken in $\M$) completes the analysis. Note that in \cite{SS00} the pushout is the free product of $T(L)$ and $X$ over $T(K)$ for the free monoid functor $T$, whereas in the commutative setting $P$ is the (conceptually simpler) tensor product $\Sym(L)\otimes_{\Sym(K)}X$. In an analogous way to the versions of these statements in \cite{SS00} we now prove
\begin{enumerate}
\item $P$ is naturally a commutative monoid.
\item $X\to P$ is a map of commutative monoids.
\item $P$ has the universal property of the pushout in the category of commutative monoids.
\end{enumerate}

As in \cite{SS00} the unit for $P$ is the map $S\to X\to P$ and the multiplication on $P$ is defined from compatible maps $P_n\otimes P_m \to P_{n+m}$ by passage to the colimit. These maps are defined inductively using the following pushout diagram (which is simply the product of the two pushout diagrams defining $P_n$ and $P_m$), where for spacing reasons we let $\tilde{Q_n}$ denote $Q_n/\Sigma_n$ and let $\tilde{L^n}$ denote $L^{\otimes n}/\Sigma_n$:
\begin{align*}
\xymatrix{(X\otimes \tilde{Q_n})\otimes (X\otimes \tilde{L^m}) \coprod_{(X\otimes \tilde{Q_n}) \otimes (X\otimes \tilde{Q_m})} (X\otimes \tilde{L^n})\otimes (X\otimes \tilde{Q_m}) \ar[d] \ar[r] & (X\otimes \tilde{L^n})\otimes (X\otimes \tilde{L^m})\ar[d]\\ (P_{n-1}\otimes P_m)\coprod_{(P_{n-1}\otimes P_{m-1})} (P_n\otimes P_{m-1}) \ar[r] & P_n\otimes P_m}
\end{align*}

This is a pushout square by Lemma 4.1 in \cite{Muro}. The lower left corner has a map to $P_{n+m}$ by induction. The upper right corner is mapped there by shuffling the middle $X$ to the left-hand side, multiplying the two factors of $X$, passing to $\Sigma_{n+m}$-coinvariants, and using the definition of $P_{n+m}$. To show $P$ is a commutative monoid one must verify the following diagrams:

\xymatrix{S\otimes P \ar[r] \ar@{=}[dr] & P\otimes P \ar[d] & P\otimes P\otimes P \ar[d]^{1\otimes \mu} \ar[r]^{\mu \otimes 1} & P\otimes P\ar[d] & P\otimes P \ar[r]^\tau \ar[d]& P\otimes P \ar[dl]\\ & P & P\otimes P\ar[r] & P & P & }

The leftmost diagram commutes because the left-hand factor of $P$ is $P_0$, coming from a map $S\to X$, and so if we replace the other factors of $P$ by $P_m$ we see that this diagram commutes before passage to colimits. In particular, the diagram defining the map $P_0\otimes P_m\to P_m$ collapses in the following way. The upper left corner is $X\otimes X\otimes L^{\otimes m}/\Sigma_m \coprod_{X\otimes X\otimes Q_m/\Sigma_m} X\otimes X \otimes Q_m/\Sigma_m= X\otimes X\otimes L^{\otimes m}/\Sigma_m$ because $X\otimes Q_0 = X$. The upper right corner is also $X\otimes X\otimes L^{\otimes m}/\Sigma_m$ because $X\otimes L^{\otimes 0}=X$. Thus, the upper horizontal map is the identity. Similarly the bottom horizontal map is the identity on $P_0\otimes P_m$. Recalling that the $P_0$ comes from a map $S\to X$ where $S$ is the monoidal unit we may write 

$$P_0\otimes P_m = (P_0\otimes P_{m-1})\coprod_{S\otimes X\otimes Q_m/\Sigma_m}(S\otimes X\otimes L^{\otimes m}/\Sigma_m) = P_{m-1}\coprod_{X\otimes Q_m/\Sigma_m}(X\otimes L^{\otimes m}/\Sigma_m) = P_m$$

Where $P_0\otimes P_{m-1}=P_{m-1}$ by induction, and the other factors of $S$ disappear because $S$ is the unit for $X$. This proves the commutativity of the leftmost diagram.


The middle diagram also commutes on the level of individual $P_i$. In particular, the two ways of getting from $P_n\otimes P_m\otimes P_k$ to $P_{n+m+k}$ (i.e. via $P_{n+m}\otimes P_k$ and via $P_n\otimes P_{m+k}$) are the same. The key observation to show this is that all maps in the diagram are of the form (Pushout $\otimes$ Identity), and the pushout of a pushout is a pushout. Thus, both ways of going around are pushouts, and the universal property of pushouts shows that they must be isomorphic.

The rightmost diagram also commutes on the level of individual $P_i$, i.e. $P_n\otimes P_m\to P_{n+m}$ is the same as $P_n\otimes P_m\to P_m\otimes P_n\to P_{m+n}$. To see this, look at the diagram defining $\mu_P$ and consider what happens if the $n$ factors and $m$ factors are swapped. This causes no harm to the upper right corner because the map from $(X\otimes L^{\otimes m}/\Sigma_m)\otimes(X\otimes L^{\otimes n}/\Sigma_n)$ requires passage to $\Sigma_{m+n}$-coinvariants, so changing the order of the $L$ factors has no effect on $\mu_P$. Similarly there is no harm to the lower left corner because of induction. The upper left corner is hardest, but either way of going around to $P_{m+n}$ will render the swapping of factors meaningless. One way around requires passage to $\Sigma_{m+n}$-coinvariants and the other way goes to $P_i\otimes P_j$ factors for $i,j<n,m$ and so will hold by induction. This completes the proof of statement (1).


To verify that the map $X\to P$ is a map of commutative monoids one must only verify that it's a map of monoids and that the two monoids in question are commutative. This means verifying the commutativity of the following diagrams:
\begin{align*}
\xymatrix{X\otimes X \ar[r] \ar[d] & X \ar[d] & S\ar[r] \ar[dr] & X\ar[d]\\ P\otimes P \ar[r] & P & & P}
\end{align*}

The map $P\otimes P\to P$ is induced by passage to colimits of the multiplication $P_n\otimes P_m\to P_{n+m}$ and so by definition the obvious diagram with $P_n\otimes P_m$, $P_{n+m}$, $P\otimes P$, and $P$ commutes for all $n,m$. The point is that defining $P\otimes P\to P$ requires one to go to $P_n\otimes P_m$, so the commutativity is tautological. In particular, it commutes for $n=m=0$ and this proves the left-hand diagram above commutes, since $X=P_0$. The right-hand diagram commutes by definition of the map $S\to P$ as coming from $X$. This completes the proof of statement (2).


To prove that $P$ satisfies the universal property of pushouts in the category of commutative monoids requires one to define a map $P\to M$ which completes the following diagram, where $M$ is a commutative monoid, $X\to M$ is monoidal, and $L\to M$ is a map in $\M$. The reason one works with $K$ and $L$ rather than $\Sym(K)$ and $\Sym(L)$ is that the data of a map of commutative monoids $\Sym(K)\to M$ is the same as that of a map from $K$ to $M$, by the free-forgetful adjunction.
\begin{align*}
\xymatrix{K\ar[r]\ar[d] & L\ar[d]\ar[ddr] & \\ X\ar[r] \ar[drr] & P\ar@{..>}[dr] & \\ & & M}
\end{align*}

The existence of maps $K\to X\to M$ and $L\to M$ defines maps from $X\otimes W(D)\to M$ for all $D$ and all $n$. Commutativity of the outer diagram forces the maps $X\otimes W(D)\to M$ to be compatible, i.e. commutativity of the square diagram featuring $X\otimes W(D), X\otimes W(D\cup \{i\}, M$, and $M$. This is because the left-vertical map in that diagram is $K\to L$ and the right vertical map is $K\to X\to M$ (which is easy to see when thinking of commutativity of the outer diagram above as defining a word in $M$). Furthermore, these maps respect the $\Sigma_n$ action on the cube defining $Q_n$ because $M$ is commutative. Thus, by induction on $n$ we may define a map $P_n\to M$ because the diagram featuring $X\otimes Q_n/\Sigma_n, X\otimes L^{\otimes n}/\Sigma_n, P_{n-1},$ and $M$ commutes. In this diagram we use induction to define the map $P_{n-1}\to M$ and we using the fact that $M$ is commutative to define the map $X\otimes L^{\otimes n}/\Sigma_n\to M$. 

Commutativity of this diagram is due to the fact that $X\otimes W(D)\to M$ factors through $X\otimes L^{\otimes |D|}/\Sigma_{|D|}$ and hence through $P_{n-1}$ via $P_{|D|}$. The unique maps $P_n\to M$ assemble to a unique map $P\to M$.

Commutativity of the triangle featuring $X,P$, and $M$ follows by definition of $P$ as a colimit and of $X$ as $P_0$. Commutativity of the other triangle follows because it holds with $P_n$ substituted for $P$, for all $n$. This is because commutativity holds in the triangle which defines the map $P_n\to M$ for all $n$, so it holds in the (first) $L$ factor of $X\otimes L^{\otimes n}/\Sigma_n$, i.e. $L\to M$ is the same as $L\to P_n\to M$ for all $n$. This completes the proof of statement (3) and hence of the proposition.
\end{proof}


We move now to homotopy theoretic considerations, and use the proposition to prove Lemma \ref{lemma:appendix-main}. 

\begin{proof}
To prove statement (1), recall that the commutative monoid axiom tells us that if $h$ is a trivial cofibration then $h^{\boxprod n}/\Sigma_n$ is a trivial cofibration for all $n>0$.

So suppose $h=j:K\stackrel{\simeq}{\hookrightarrow} L$. Because $j$ is a trivial cofibration, the map $j^{\boxprod n}/\Sigma_n:Q_n/\Sigma_n \to \Sym^n(L)$ is a trivial cofibration. Thus, the map $X\otimes j^{\boxprod n}/\Sigma_n$ is of the form required by the monoid axiom. This means transfinite compositions of pushouts of such maps are weak equivalences, so in particular $X\to P$ is a weak equivalence in $\M$ and hence in $\CMon(\M)$. Any map in $\Sym(J)$-cell is a transfinite composite of pushouts of maps in $\Sym(J)$. We have seen that all such pushouts are of the form required by the monoid axiom, and a transfinite composite of a transfinite composite is still a transfinite composite, so the monoid axiom applied again proves that $\Sym(J)$-cell is contained in the weak equivalences. This completes the proof of (1).

For (2), suppose $h=i:K\hookrightarrow L$ and suppose $X$ is cofibrant in $\M$. By the strong commutative monoid axiom, the maps $i^{\boxprod n}/\Sigma_n$ are cofibrations for all $n$, so $X\otimes i^{\boxprod n}/\Sigma_n$ are cofibrations for all $n$. Since pushouts of cofibrations are again cofibrations, the maps $P_{n-1}\to P_n$ are cofibrations for all $i$. Because $P_0=X$ is cofibrant, this means all the $P_k$ are cofibrant and also $X\to P$ is a cofibration (so $P$ is cofibrant) because transfinite compositions of cofibrations are again cofibrations (see Proposition 10.3.4 in \cite{hirschhorn}). Every map in $\Sym(I)$-cell which has cofibrant domain is a transfinite composite of pushouts of maps of the form above, and so is in particular again a cofibration in $\M$.

\end{proof}

\section{Operadic Generalization} \label{sec:operads}

In the proof above, we make use of a particular filtration on the map $X\to P$. We could also have followed \cite{DAG3} and filtered the map $\Sym(f)$ as 
$$\Sym(K) = B_0 \to B_1\to \dots \to \Sym(L)$$
where each $B_n$ is a $\Sym(K)$-module. This makes it clear that the map $X\to P$ is a map of $X$-modules, and thus makes it easier to check that $P$ is in fact a monoid. However, this filtration requires special knowledge of $Com$, namely that it is generated by $Com(2)$-swaps (i.e. functions of arity two) so that $Com$-algebras can be multiplied with themselves. The author chose the approach presented here because it allows for an easy generalization to operads. 

The commutative monoid axiom has a natural generalization to an arbitrary operad $P$. Recall that cofibrancy may be defined for $P$-algebras via a lifting property, even if the category of $P$-algebras is not a model category.

\begin{defn} \label{defn:general-noncofibrant}
Let $P$ be an operad. A monoidal model category $\M$ is said to satisfy the \textit{$P$-algebra axiom} if for all cofibrant $P$-algebras $A$ and for all $n\geq 0$, $P_A(n) \otimes_{\Sigma_n} (-)^{\boxprod n}$ preserves trivial cofibrations (where $P_A$ is the enveloping operad).
\end{defn}

\begin{theorem} \label{thm:semi-on-P-alg}
Let $P$ be an operad (always assumed symmetric) and suppose $\M$ is a combinatorial model category satisfying the $P$-algebra axiom. Then the category $P$-alg$(\M)$ inherits a semi-model structure from $\M$.
\end{theorem}

\begin{proof}
As usual, this semi-model structure will be transferred along the free-forgetful adjunction $(P,U)$ via Lemma \ref{ss00-lemma2.3}. Because $\M$ is combinatorial, the smallness hypotheses of Lemma \ref{ss00-lemma2.3} are automatically satisfied. Let $j:K\to L$ be a trivial cofibration in $\M$ and consider the pushout of $P(j)$ along a $P$-algebra homomorphism $P(K)\to A$. Denote the resulting map $\gamma:A\to B$. Factor this map as in the Section 7.3 of \cite{harper-operads}, recalled in Remark \ref{remark-enveloping-computations}. So $\gamma$ is a transfinite composite of pushouts of maps of the form $P_A(n) \otimes_{\Sigma_n} j^{\boxprod n}$. By hypothesis all such maps are trivial cofibrations, so $\gamma$ is a trivial cofibration. As in Corollary \ref{cor:semi-on-comm-alg}, this completes the proof.
\end{proof}

While the $P$-algebra axiom gives a minimal condition on $\M$ so that $P$-algebras inherit a semi-model structure, it is not clear that this condition can be checked in practice because of the presence of $P_A$ in the hypotheses. However, we can generalize the commutative monoid axiom to find a new family of axioms on $\M$ which do not make reference to $P_A$. This line of reasoning has been used in \cite{spitzweck-thesis} (where conditions are given so that $\Sigma$-cofibrant operads are admissible) and in \cite{harper-operads} (where conditions are given so that all operads are admissible).

These two examples demonstrate that in order for the category of $P$-algebras to inherit a semi-model structure, a cofibrancy hypothesis on either $\M$ or $P$ will be needed. The following result will unify all previous results on admissibility into a single framework and provide new results for levelwise cofibrant operads, where the cofibrancy hypotheses are evenly distributed between the operad and the model category. 

\begin{theorem} \label{thm:operad-table}
Let $\M$ be a combinatorial monoidal model category. Let $f$ run through the class of (trivial) cofibrations. Consider the following hypothesis, where $X$ is an object with a $\Sigma_n$-action that runs through some class of objects $\mathcal{K}$:

\underline{Hypothesis}: $X\otimes_{\Sigma_n}f^{\boxprod n}$ is a (trivial) cofibration for all $X \in \mathcal{K}$.

In each row of the following table, placing this hypothesis on $\M$ for the class of objects $\mathcal{K}$ listed in the left column gives a semi-model structure on $P$-algebras for all $P$ satisfying the hypotheses in the right column.
\end{theorem}
\[
\begin{tabular}{c|c}
Hypothesis on $\M$ & Class of operad\\
\hline
$\mathcal{K}=\{\Sigma_n-$projectively cofibrant objects $\}$& $(\Sigma-)$Cofibrant\\
\hline
$\mathcal{K}=\{$objects cofibrant in $\M \}$ & Levelwise cofibrant\\
\hline
$\mathcal{K}=\{$objects in $\M^{\Sigma_n} \}$ & Arbitrary\\
\end{tabular}
\]

The hypotheses going down the left column are increasing in strength, while the hypotheses in the right column are decreasing. The last row says that if $\M$ is combinatorial, monoidal, satisfies the monoid axiom, and has the property that $\forall X \in \M^{\Sigma_n}$, $X\otimes_{\Sigma_n}f^{\boxprod n}$ is a (trivial) cofibration, then all operads are admissible. This generalizes the main theorem from \cite{harper-operads}, which states that if all symmetric sequences in $\M$ are projectively cofibrant then all operads are admissible. Similarly, the first row recovers a theorem of Spitzweck from \cite{spitzweck-thesis}, since it follows from the pushout product axiom that for any $\Sigma_n$-projectively cofibrant $X$, the map $X\otimes_{\Sigma_n}f^{\boxprod n}$ is a trivial cofibration. The row regarding levelwise cofibrant operads is new. 

\begin{proof}
The proof proceeds as in Remark \ref{remark-enveloping-computations}, but using $P_A(n)$ rather than $Com_A(n)$. The hypothesis in the theorem guarantees this procedure will work as soon as $P_A(n)$ is known to be in the class of objects considered in the left-hand column. For the bottom row this condition is automatic. For the top row one may use Proposition 5.17 in \cite{harper-hess}. For the middle row, we must show $P_A(n)$ is cofibrant in $\M$ if $P$ is levelwise cofibrant and $A$ is a cofibrant $P$-algebra. The proof in Proposition 5.17 (and Proposition 5.44a, on which it relies) in \cite{harper-hess} goes through mutatis mutandis.
\end{proof}

The author is still working to reduce the hypotheses on $\M$ so that combinatoriality is not required. This will come down to better understanding what the free $P$-algebra functor does to the domains of the generating trivial cofibrations. The interested reader may fill in an appropriately weakened smallness hypothesis on the domains. The author, together with Donald Yau, has worked out in \cite{white-yau} a generalization to this theorem in the setting of colored operads. This involves generalizing the filtration of Remark \ref{remark-enveloping-computations} to the colored setting.

These hypotheses on $\M$ are not too difficult to check. For example, the one for levelwise cofibrant operads holds for $sSet$, even though the hypothesis in the bottom row does not hold for $sSet$. The bottom row holds for $Ch(k)$ for $k$ a field of characteristic 0 and for the positive flat model structure on symmetric spectra (by arguments analogous to those found in Section \ref{sec:examples} above).

To get from a semi-model structure to a full model structure we would need to add a new hypothesis on $\M$. By way of analogy, note that to do this for cofibrant operads or for $Com$, the monoid axiom was needed. This is because the filtration by $P_A$ is simpler in these cases. In general, we need a hypothesis similar to the monoid axiom but which takes the $\Sigma_n$ action into account. 

\begin{defn}
Let $\sQ_{\Sigma_n}^t$ be the class of maps in $\M^{\Sigma_n}$ which are trivial cofibrations in $\M$. We say $\M$ satisfies the \textit{$\Sigma_n$-equivariant monoid axiom} if transfinite compositions of pushouts of maps of the form $(\sQ^t_{\Sigma_n}) \otimes_{\Sigma_n} X$ are contained in the weak equivalences for all $X \in \M^{\Sigma_n}$.
\end{defn}

It is clear from the filtration argument given in Section 7.3 in \cite{harper-operads} that this hypothesis will imply the semi-model structures are actually model structures. However, this hypothesis is in fact so strong that it alone proves all operads are admissible, regardless of the hypotheses in Theorem \ref{thm:operad-table}. We summarize:

\begin{corollary}
Suppose $\M$ is a combinatorial monoidal model category satisfying the $\Sigma_n$-equivariant monoid axiom. Then for any operad $P$, algebras over $P$ inherit a model structure from $\M$.
\end{corollary}

A simpler hypothesis to check, which also works to improve a semi-model structure to a model structure, is the hypothesis that all objects are cofibrant. Combined with our earlier observation about $sSet$ this implies all levelwise cofibrant operads (hence all operads) are admissible when $\M=sSet$.

\section*{Acknowledgments}

The author would like to gratefully acknowledge the support and guidance of his advisor Mark Hovey as this work was completed. The author is also indebted to John Harper, Luis Pereira, Marcy Robertson, Dror Farjoun, Carles Casacuberta, Brooke Shipley, Michael Batanin, and Richard Garner for many helpful conversations. The author also thanks Andrew Blumberg for catching an error in an early version of this paper, Birgit Richter for helpful remarks, and the anonymous referee for numerous helpful suggestions.

\end{document}